\documentclass[11pt,a4paper,reqno]{amsart}
\usepackage[english]{babel}
\usepackage[T1]{fontenc}
\usepackage{verbatim}
\usepackage{palatino}
\usepackage{amsmath}
\usepackage{mathabx}
\usepackage{amssymb}
\usepackage{amsthm}
\usepackage{amsfonts}
\usepackage{graphicx}
\usepackage{esint}
\usepackage{xcolor}
\usepackage{mathtools}
\usepackage{overpic}
\usepackage{enumerate}
\DeclarePairedDelimiter\ceil{\lceil}{\rceil}

\usepackage[colorlinks = true, linkcolor={red!80!black}, citecolor = {blue!60!black}]{hyperref}
\author{Tuomas Orponen, Pablo Shmerkin and Hong Wang}
\title[Radial projections]{Kaufman and Falconer estimates for radial projections \\ and a continuum version of Beck's Theorem}
\address{Department of Mathematics and Statistics\\ University of Jyv\"askyl\"a,
P.O. Box 35 (MaD)\\
FI-40014 University of Jyv\"askyl\"a\\
Finland} \email{tuomas.t.orponen@jyu.fi}
\address{Department of Mathematics\\The University of British Columbia,
1984 Mathematics Road\\
Vancouver, BC, V6T 1Z2\\
Canada } \email{pshmerkin@math.ubc.ca}
\address{Department of Mathematics, UCLA, Los Angeles, CA 90095, USA} \email{hongwang@math.ucla.edu}
\date{\today}
\thanks{T.O. is
 supported by the Academy of Finland via the projects
\emph{Quantitative rectifiability in Euclidean and non-Euclidean
spaces} and \emph{Incidences on Fractals}, grant Nos. 309365,
314172, 321896. }
\thanks{P.S. is supported by an NSERC Discovery Grant}
\thanks{H.W. is supported by NSF Grant DMS-2055544}

\newcommand{\R}{\mathbb{R}}

\newcommand{\N}{\mathbb{N}}

\newcommand{\spt}{\operatorname{spt}}
\newcommand{\Hd}{\dim_{\mathrm{H}}}

\newcommand{\spa}{\operatorname{span}}

\newcommand{\dist}{\operatorname{dist}}

\newcommand{\e}{\epsilon}

\def\Barint_#1{\mathchoice
          {\mathop{\vrule width 6pt height 3 pt depth -2.5pt
                  \kern -8pt \intop}\nolimits_{#1}}%
          {\mathop{\vrule width 5pt height 3 pt depth -2.6pt
                  \kern -6pt \intop}\nolimits_{#1}}%
          {\mathop{\vrule width 5pt height 3 pt depth -2.6pt
                  \kern -6pt \intop}\nolimits_{#1}}%
          {\mathop{\vrule width 5pt height 3 pt depth -2.6pt
                  \kern -6pt \intop}\nolimits_{#1}}}

\numberwithin{equation}{section}

\theoremstyle{plain}
\newtheorem{thm}[equation]{Theorem}
\newtheorem*{"thm"}{"Theorem"}
\newtheorem{conjecture}[equation]{Conjecture}
\newtheorem{lemma}[equation]{Lemma}

\newtheorem{cor}[equation]{Corollary}
\newtheorem{proposition}[equation]{Proposition}

\theoremstyle{definition}

\newtheorem{definition}[equation]{Definition}

\theoremstyle{remark}
\newtheorem{remark}[equation]{Remark}

\newcommand{\nref}[1]{(\hyperref[#1]{#1})}

\DeclareMathSymbol{\intop}  {\mathop}{mathx}{"B3}

\begin{document}

\begin{abstract} We provide several new answers on the question: how do radial projections distort the dimension of planar sets?  Let $X,Y \subset \mathbb{R}^{2}$ be non-empty Borel sets. If $X$ is not contained on any line, we prove that
\[
\sup_{x \in X} \dim_{\mathrm{H}} \pi_{x}(Y \, \setminus \, \{x\}) \geq \min\{\dim_{\mathrm{H}} X,\dim_{\mathrm{H}} Y,1\}.
\]
If $\dim_{\mathrm{H}} Y > 1$, we have the following improved lower bound:
\[ 
\sup_{x \in X} \dim_{\mathrm{H}} \pi_{x}(Y \, \setminus \, \{x\}) \geq \min\{\dim_{\mathrm{H}} X + \dim_{\mathrm{H}} Y - 1,1\}.
\]
Our results solve conjectures of Lund-Thang-Huong, Liu, and the first author. Another corollary is the following continuum version of Beck's theorem in combinatorial geometry: if $X \subset \mathbb{R}^{2}$ is a Borel set with the property that $\dim_{\mathrm{H}} (X \, \setminus \, \ell) = \dim_{\mathrm{H}} X$ for all lines $\ell \subset \mathbb{R}^{2}$, then the line set spanned by $X$ has Hausdorff dimension at least $\min\{2\dim_{\mathrm{H}} X,2\}$.

While the results above concern $\mathbb{R}^{2}$, we also derive some counterparts in $\mathbb{R}^{d}$ by means of integralgeometric considerations. The proofs are based on an $\epsilon$-improvement in the Furstenberg set problem, due to the two first authors, a bootstrapping scheme introduced by the second and third author, and a new planar incidence estimate due to Fu and Ren. \end{abstract}

\maketitle

\tableofcontents

\section{Introduction}

We study the distortion of Hausdorff dimension under \emph{radial projections} $\pi_{x} \colon \R^{d} \, \setminus \, \{x\} \to S^{d - 1}$ defined by
\begin{displaymath} \pi_{x}(y) := \frac{y - x}{|y - x|}, \qquad y \in \R^{d} \, \setminus \, \{x\}. \end{displaymath}
Here are the main results in the plane:
\begin{thm}\label{mainNew} Let $X \subset \R^{2}$ be a (non-empty) Borel set which is not contained on any line. Then, for every Borel set $Y \subset \R^{2}$,
\begin{displaymath} \sup_{x \in X} \Hd \pi_{x}(Y \, \setminus \, \{x\}) \geq \min\{\Hd X,\Hd Y,1\}. \end{displaymath}
\end{thm}
In the case $\Hd Y > 1$, we have the following improved lower bound:

\begin{thm}\label{main} Let $X,Y \subset \R^{2}$ be Borel sets with $X \neq \emptyset$ and $\Hd Y > 1$. Then,
\begin{displaymath} \sup_{x \in X} \Hd \pi_{x}(Y \, \setminus \, \{x\}) \geq \min\{\Hd X + \Hd Y - 1,1\}. \end{displaymath}
\end{thm}
We mention a few corollaries of Theorem \ref{mainNew}. First, it can be applied to the case $Y = X$, assuming that $X$ is a Borel set which does not lie on a line. Then,
\begin{displaymath} \sup_{x \in X} \Hd \pi_{x}(X \, \setminus \, \{x\}) = \min\{\Hd X,1\}. \end{displaymath}
In particular, the \emph{direction set} $S(X) := \{(x - y)/|x - y| : x,y \in X, \, x \neq y\}$ has
\[
\Hd S(X) \geq \min\{\Hd X,1\}.
\]
This solves \cite[Conjecture 1.9]{MR3892404}. These results were earlier proved under the additional assumption that $X$ has equal Hausdorff and packing dimension by the second and third author in \cite[Theorem 1.6]{2021arXiv211209044S}. Weaker bounds on the dimension of $\pi_{x}(X \, \setminus \, \{x\})$ for sets $X$ not contained in a line were previously obtained in \cite{MR3892404, Shmerkin20, LiuShen20}.

As a second application of Theorem \ref{mainNew}, we solve \cite[Conjecture 1.2]{MR4269398} in $\R^{2}$:
\begin{cor}\label{main3} Let $Y \subset \R^{2}$ be a Borel set with $\Hd Y \leq 1$. Then,
\begin{displaymath} \Hd \{x \in \R^{2} \, \setminus \, Y : \Hd \pi_{x}(Y) < \Hd Y\} \leq 1. \end{displaymath}
\end{cor}

\begin{proof} Assume to the contrary that there exists $\epsilon > 0$ such that
\begin{displaymath} X = \{x \in \R^{2} : \Hd \pi_{x}(Y \, \setminus \, \{x\}) \leq \Hd Y - \epsilon\} \end{displaymath}
has $\Hd X > 1$. Then $X$ evidently does not lie on a line. Now, Theorem \ref{mainNew} tells us that there exists a point $x \in X$ such that $\Hd \pi_{x}(Y \, \setminus \, \{x\}) > \Hd Y - \epsilon$, a contradiction. \end{proof}

Finally, Theorem \ref{mainNew} yields a continuum version of \emph{Beck's theorem} \cite{MR729781} from 1983. The original version states that if $P \subset \R^{2}$ is a finite set of points, then either $\gtrsim |P|$ points lie on a single line, or else $P$ spans $\gtrsim |P|^{2}$ distinct lines. This is a simple consequence of the Szemer\'edi-Trotter incidence bound \cite{MR729791}. The continuum version is the following:

\begin{cor} Let $X \subset \R^{2}$ be a Borel set such that $\Hd (X \, \setminus \, \ell) = \Hd X$ for all lines $\ell \subset \R^{2}$. Then, the line set $\mathcal{L}(X)$ spanned by pairs of (distinct) points in $X$ satisfies
\begin{displaymath} \Hd \mathcal{L}(X) \geq \min\{2\Hd X,2\}. \end{displaymath}
\end{cor}

\begin{proof} Let $B := \{x \in X : \Hd \pi_{x}(X \, \setminus \, \{x\}) < \min\{\Hd X,1\}\}$. We claim that
\begin{equation}\label{form70} \Hd (X \, \setminus \, B) = \Hd X. \end{equation}
In fact, this will complete the proof: by definition, for each $x \in X \, \setminus \, B$ there exists a line set $\mathcal{L}_{x} \subset \mathcal{L}(X)$ which contains $x$ and has $\Hd \mathcal{L}_{x} = \min\{\Hd X,1\} =: s$. Therefore
\begin{displaymath} \mathcal{L}(X) \supset \bigcup_{x \in X \, \setminus \, B} \mathcal{L}_{x} \end{displaymath}
contains a (dual) $(s,s)$-Furstenberg set, and thus $\Hd \mathcal{L}(X) \geq 2s$, see \cite[Theorem A.1]{HSY21}.

Let us prove \eqref{form70}. If \eqref{form70} fails, then clearly $\Hd B = \Hd X$. Then $B$ must lie on a line, say $B \subset \ell$. Otherwise Theorem \ref{mainNew} tells us that there exists a point $x \in B$ such that $\Hd \pi_{x}(X \, \setminus \, \{x\}) = \min\{\Hd X,1\}$, contrary to the definition of $B$. But now $\Hd (X \, \setminus \, B) \geq \Hd (X \, \setminus \, \ell) = \Hd X$ by assumption, so actually \eqref{form70} holds. \end{proof}

We then turn to Theorem \ref{main}. It solves the planar case of \cite[Conjecture 1.2]{2022arXiv220507431L}. In fact, the numerology may be more recognisable if we restate Theorem \ref{main} as follows:
\begin{cor}\label{mainRestatement} Let $Y \subset \R^{2}$ be a Borel set with $\Hd Y > 1$. Then,
\begin{displaymath} \Hd \{x \in \R^{2} : \Hd \pi_{x}(Y \, \setminus \, \{x\}) < \sigma\} \leq \max\{1 + \sigma - \Hd Y,0\}, \qquad 0 \leq \sigma \leq 1. \end{displaymath}
\end{cor}

\begin{proof} Assume to the contrary that the exceptional set "$X$" on the left has dimension $\Hd X > \max\{1 + \sigma - \Hd Y,0\}$. In particular $X \neq \emptyset$. Then, by Theorem \ref{main},
\begin{displaymath} \sup_{x \in X} \Hd \pi_{x}(Y \, \setminus \, \{x\}) \geq \min\{\Hd X + \Hd Y - 1,1\} \geq \sigma, \end{displaymath}
which is a contradiction.
\end{proof}

As is often the case, we deduce our main results from corresponding quantitative, discretised versions: they are Corollary \ref{cor:radial} for Theorem \ref{mainNew} (see also Corollary \ref{cor:radial2}), and Theorem  \ref{mainThinTubes}  for  Theorem \ref{main}.

\subsection{Related work and higher dimensions}\label{s:relatedWork} Proving Corollary \ref{main3} does not require the full strength of Theorem \ref{mainNew}. In fact, it could also be deduced from (a quantitative version) of Theorem \ref{main}, combined with a "swapping trick" introduced by Liu in \cite{MR4269398}. The reason is partially visible from the proof of Corollary \ref{main3}: one is allowed to make the counter assumption that the "exceptional set" $X$ has $\Hd X > 1$, and this simplifies matters (the hardest case of Theorem \ref{mainNew} occurs when $\Hd X,\Hd Y \leq 1$). This approach of deducing Corollary \ref{main3} from Theorem \ref{main} is carried out in a prequel to this paper \cite{2022arXiv220513890O} (where a proof of Theorem \ref{main} first appeared - the proof in the current paper is conceptually similar, but the technical implementation is simplified). Thus, at the level of appropriate quantitative versions of the statements, both Theorems \ref{main} and \ref{mainNew} imply Corollary \ref{main3}, but neither of these theorems imply each other (as far as we know).

After \cite{2022arXiv220513890O} appeared on the \emph{arXiv}, Dote and Gan \cite{2022arXiv220803597D} proved a higher dimensional version of Theorem \ref{main}, and then used the "swapping trick" to obtain a higher dimensional counterpart of Corollary \ref{main3}. Their results are the following:
\begin{thm}[Dote-Gan]\label{ganDote1} Let $Y \subset \R^{d}$ be a Borel set with $\Hd Y \in (k,k + 1]$ for some $k \in \{1,\ldots,d - 1\}$. Then,
\begin{displaymath} \Hd \{x \in \R^{d} : \Hd \pi_{x}(Y \, \setminus \, \{x\}) < \sigma\} \leq \max\{k + \sigma - \Hd Y\}, \qquad 0 \leq \sigma \leq k. \end{displaymath}
 \end{thm}
\begin{thm}[Dote-Gan]\label{ganDote2} Let $Y \subset \R^{d}$ be a Borel set with $\Hd Y \in (k - 1,k]$ for some $k \in \{1,\ldots,d - 1\}$. Then,
\begin{displaymath} \Hd \{x \in \R^{d} : \Hd \pi_{x}(Y \, \setminus \, \{x\}) < \Hd Y\} \leq k. \end{displaymath} \end{thm}

We give new proofs for Theorems \ref{ganDote1}-\ref{ganDote2} at the end of this paper, Section \ref{s:higherDim}. In fact, both statements can be deduced from (quantitative versions of) their planar cases via an integralgeometric argument (due to the second and third author in \cite[Theorem 6.7]{2021arXiv211209044S}). Moreover, we are able to prove a partial higher-dimensional analogue of Theorem \ref{mainNew} that implies Theorem \ref{ganDote2} as a by-product:
\begin{thm}\label{mainHighDim}
Let $X,Y\subset\R^d$, $d\ge 2$, be Borel sets with $\Hd X > k - 1$ and $\Hd Y \in (k - 1,k]$ for some $k\in\{1,\ldots,d-1\}$.
\begin{enumerate}
\item[\textup{(i)}] If $\Hd X > k$, then $\sup_{x \in X} \Hd \pi_{x}(Y \, \setminus \, \{x\}) = \Hd Y$.
\item[\textup{(ii)}] If $k - 1 < \Hd X \leq k$, but $X$ is not contained on any $k$-plane, the following holds. If $\Hd Y > k - 1/k - \eta$ for a sufficiently small constant $\eta = \eta(d,k,\Hd X) > 0$, then
\begin{displaymath} \sup_{x \in X} \Hd \pi_{x}(Y \, \setminus \, \{x\}) \geq \min\{\Hd X,\Hd Y\}. \end{displaymath}
\end{enumerate}
For $k = 1$, we require no lower bound from $\Hd Y$ in part \textup{(ii)}.
\end{thm}

Part (i) is just a restatement of Theorem \ref{ganDote2}. The lower bound $\Hd Y > k-1/k-\eta$ in part (ii) may appear odd, and is likely an artefact of the proof. We conjecture that it can be replaced by $\Hd Y > k-1$. Our argument would yield this if the following was true:
\begin{conjecture} \label{conj:cond-meas}
Let $t \in (d - 2,d]$. Let $\mu$ be a Borel probability measure on $\R^d$ that satisfies $\mu(B(x,r))\le C\, r^t$ for all $x \in \R^{d}$, $r > 0$, and some $C>0$. Further, assume that $\mu(W)=0$ for every affine $(d - 1)$-plane $W\subset\R^d$. Then, for almost every affine $2$-plane $W \subset \R^{d}$ (with respect to the natural measure on the affine Grassmanian), if the sliced measure $\mu_{W}$ on $W$ is non-trivial, then it does not give full mass to any line.
\end{conjecture}
We refer to  \cite[Chapter 10]{zbMATH01249699} or \cite[Chapter 6.1]{Mattila19} for the definition of the measures $\mu_{W}$. In \cite[Proposition 6.8]{2021arXiv211209044S}, the last two authors proved a weaker statement under the assumption $t>d-1-1/(d-1)-\eta(d)$. This explains the numerology in Theorem \ref{mainHighDim}.

There are plenty of earlier relevant results on radial projections; sometimes this topic is also studied under the name \emph{visibility}. The finite field counterparts of Theorems \ref{ganDote1}-\ref{ganDote2} were proven a little earlier by Lund, Thang, and Huong Thu in \cite{2022arXiv220507431L}. This is also where the continuum version of Theorem \ref{ganDote1} was conjectured. We are not aware of a finite field counterpart to Theorem \ref{mainNew}. As we already mentioned, the continuum version of Theorem \ref{ganDote2} was conjectured by Liu \cite{MR4269398}, who also proved partial results. The special case $k = d - 1 = \sigma$ of Theorem \ref{ganDote1} was contained in \cite[Theorem 1.1]{MR3778538}. Partial results also follow from Peres and Schlag's general framework of \emph{transversal projections} \cite{MR1749437}.

In 2021, Raz and Zahl \cite[Theorem 1.13]{2021arXiv210807311R} proved a radial projection theorem which gives non-trivial information if the set of "viewpoints" (the set "$X$" in the results described above) only has $4$ elements, all $3$ of which span a non-degenerate triangle. Earlier, in \cite[Theorem 1.6]{MR3458388}, Bond, {\L}aba and Zahl and proved "single-scale" estimates for radial projections of planar sets which are \emph{unconcentrated on lines}. This result is in the spirit of Theorem \ref{mainNew}, and perhaps its earliest precedent in the literature.

A little further away from the topic of this paper, radial projections and visibility have also been investigated in the context of rectifiability. The heuristic is that purely $1$-unrectifiable sets $Y \subset \R^{2}$ with $\mathcal{H}^{1}(Y) < \infty$ should have radial projections of zero length for "most" viewpoints $x \in \R^{2}$. Marstrand \cite[Theorem VI]{Mar} already proved that if $Y$ is as above, then $\mathcal{H}^{1}(\pi_{x}(Y \, \setminus \, \{x\})) = 0$ for all $x \in \R^{2} \, \setminus \, X$, where $\Hd X \leq 1$. Whether the same is true for $\mathcal{H}^{1}$ almost all $x \in Y$ is a well-known open problem. For more information on this question, and related ones, see \cite{MR1762426,MR2329222} and \cite[Section 6]{MR2044636}.

\subsection{Connection with orthogonal projections} Theorems \ref{mainNew} and \ref{main} can be viewed as stronger versions of classical results in fractal geometry regarding \emph{orthogonal projections}. For $e \in S^{1}$, let $P_{e}(x) := (x \cdot e)e$ be the orthogonal projection to the line spanned by $e$. If $Y \subset \R^{2}$ is a Borel set, then Kaufman \cite{Ka} in 1968 proved that
\begin{equation}
\label{kaufman} \Hd \{e \in S^{1} : \Hd P_{e}(Y) < \sigma\} \leq \sigma, \qquad 0 \leq \sigma \leq \min\{\Hd Y,1\}.\\
\end{equation}
If $\Hd Y > 1$, then the following improved estimate holds:
\begin{equation}
\label{peresSchlag} \Hd \{e \in S^{1} : \Hd P_{e}(Y) < \sigma\} \leq \max\{1 + \sigma - \Hd Y,0\}, \quad 0 \leq \sigma \leq 1.
\end{equation}
Estimate \eqref{peresSchlag} is due to Peres and Schlag \cite{MR1749437}, but the special case $\sigma = 1$ was proven already in 1982 by Falconer \cite{MR673510}. The proof of \eqref{peresSchlag} is based on Fourier analysis. The proofs of \eqref{kaufman}-\eqref{peresSchlag} can be conveniently read from Mattila's book \cite[Chapter 5]{MR3617376}.

Let us then explain the relationship between Theorems \ref{mainNew}-\ref{main} and the bounds \eqref{kaufman}-\eqref{peresSchlag}. If $\ell \subset \R^{2}$ is a line, then there exists a projective transformation $F_{\ell}$ with the property
\begin{displaymath} \Hd \pi_{x}(Y) = \Hd P_{e(x)}(F_{\ell}(Y)), \qquad x \in \ell, \, Y \subset \R^{2} \, \setminus \, \ell, \end{displaymath}
where the map $x \mapsto e(x) \in S^{1}$ is locally bilipschitz (and in particular preserves Hausdorff dimension). For more details, see Remark \ref{rem1}. In particular, the estimates \eqref{kaufman}-\eqref{peresSchlag} hold as stated for radial projections to points on "$\ell$". For example, \eqref{peresSchlag} yields
\begin{equation}\label{radialFalconer} \Hd \{x \in \ell : \Hd \pi_{x}(Y \, \setminus \, \{x\}) < \sigma \} \leq \max\{1 + \sigma - \Hd Y,0\}, \qquad 0 \leq \sigma \leq 1, \end{equation}
assuming that $\Hd Y > 1$. Since $\R^{2}$ is foliated by parallel lines, one might guess by a heuristic "Fubini argument" that
\begin{displaymath} \Hd \{x \in \R^{2} : \Hd \pi_{x}(Y \, \setminus \, \{x\}) < \sigma \} \leq \max\{2 + \sigma - \Hd Y,0\}, \qquad 0 \leq \sigma \leq 1. \end{displaymath}
This is not rigorous, but nevertheless the bound above follows from Peres and Schlag's general theory of \emph{transversal projections} \cite[Theorem 7.3]{MR1749437} applied to radial projections. Now, compare these estimates with Corollary \ref{mainRestatement}, which was a restatement of Theorem \ref{main}. Corollary \ref{mainRestatement} allows one to replace "$\ell$" by "$\R^{2}$" in \eqref{radialFalconer} while keeping the right hand side unchanged. So, the "Fubini argument" is extremely unsharp!

The Kaufman estimate \eqref{kaufman} bears a similar relationship to Theorem \ref{mainNew}. To see this, let us begin by mentioning that \eqref{kaufman} is formally equivalent to the following: if $\emptyset \neq X \subset S^{1}$ and $Y \subset \R^{2}$ are Borel sets, then $\sup_{e \in X} \Hd P_{e}(Y) \geq \min\{\Hd X,\Hd Y,1\}$. This version looks more like Theorem \ref{mainNew}. Moreover, by applying the projective transformation $F_{\ell}$ as above, one could deduce the following corollary for every fixed line $\ell \subset \R^{2}$: if $\emptyset \neq X \subset \ell$ and $Y \subset \R^{2} \, \setminus \, \ell$ are Borel sets, then
\begin{equation}\label{form72} \sup_{x \in X} \Hd \pi_{x}(Y) \geq \min\{\Hd X,\Hd Y,1\}. \end{equation}
Theorem \ref{mainNew} almost looks like \eqref{form72} without the restrictions "$X \subset \ell$" and "$Y \subset \R^{2} \, \setminus \, \ell$". The only problem is that such a statement is completely false: if $X,Y$ happen to lie on a common line, then $\Hd \pi_{x}(Y \, \setminus \, \{x\}) = 0$ for all $x \in X$. Theorem \ref{mainNew} takes this obstruction into account by assuming that $X$ is not contained in any line.

\subsection{Some words on the proofs}

The starting point of this work was the observation that while Beck's Theorem relies on the Szemer\'{e}di-Trotter incidence bound, in fact it doesn't require the full strength of Szemer\'{e}di-Trotter - any "$\epsilon$-improvement" over the elementary double-counting bound is enough. Recently, the first two authors \cite{2021arXiv210603338O} obtained an $\epsilon$-improvement in the \emph{Furstenberg set problem}, which can be seen as a continuous analogue of the point-line incidence problem from discrete geometry. (The actual discretised result we use is stated below as Theorem \ref{thm:improved-Furstenberg}.) As it is often the case in this area, while the discrete result (in this case, Beck's Theorem and its proof) provided the inspiration for our work, the proof is substantially more involved than in the discrete case.

Our general strategy is to embed the $\epsilon$-improvement on the Furstenberg set problem into  a "bootstrapping" argument where one gradually improves the lower bound in Theorem \ref{mainNew}, up to the threshold $\min\{\Hd X,\Hd Y,1\}$. The main work is contained in Lemma \ref{lem: radial2} and Corollary \ref{cor:radial}. A similar bootstrapping scheme had appeared in previous work of the second and third author, see \cite[Lemmas 5.11 and 5.17]{2021arXiv211209044S}. However, the proof of the bootstrapping step in \cite{2021arXiv211209044S} relied on a linearization argument that only yields optimal results for sets with additional regularity, such as equal Hausdorff and packing dimension. Hence, the main innovation of this article lies in the proof of the bootstrapping step (Lemma \ref{lem: radial2}). Arguing by contradiction under the assumption that a lower bound $\sigma<\min\{\Hd X,\Hd Y,1\}$ cannot be improved, we are able to pigeonhole a small scale $0< r\ll 1$ and a family of "well-behaved" $r$-tubes $\mathcal{T}$ through $X$ and $Y$ with $|\mathcal{T}|\ge r^{-2\sigma-\eta}$, where $\eta>0$ ultimately comes from the $\epsilon$-improvement in \cite{2021arXiv210603338O}. This allows us to view $Y$ as a sort of discretised Furstenberg-type set at scale $r$. If the mass of $Y\cap T$ for a typical $r$-tube in $\mathcal{T}$ is not too concentrated inside a ball of radius $r^\kappa$ (where $\kappa$ is a carefully chosen parameter, depending on $\eta$), then we can use a variant of the classical "two-ends" argument for Furstenberg sets to derive a contradiction (arising from $\eta>0$). On the other hand, if $Y\cap T$ is often concentrated on an $r^\kappa$-ball, we can use a double-counting argument to show that for many $y\in Y$ there will be many tubes $T$ through $y$ with this property, and this leads to the absurd conclusion that the ball $B(y,r^\xi)$ has too large mass to be compatible with $\sigma<\Hd(Y)$.

The proof of Theorem \ref{main} is based on a reduction to a recent incidence estimate of Fu and Ren \cite{2021arXiv211105093F}, combined with some elementary estimates on \emph{Furstenberg sets} (these will be discussed later). The proof of Fu and Ren, further, involves a Fourier-analytic component, due to Guth, Solomon, and Wang \cite{GSW19}. So, while this paper contains no Fourier transforms, they play a role in the proof of Theorem \ref{main}. As we already mentioned, the higher-dimensional counterpart, Theorem \ref{ganDote1}, can be deduced from (a quantitative version) of the planar case via integralgeometric considerations, see Section \ref{s:higherDim}. This is not the approach of Dote and Gan: in \cite{2022arXiv220803597D}, Theorem \ref{ganDote1} is proved more directly in general dimensions, although the proof still involves Fourier analysis.

\subsection{Connections and applications}

Part of the impetus for the study of radial projections in recent years came from the realization that they are closely connected to the Falconer distance set problem. In particular, a radial projection theorem of the first author \cite[Theorem 1.11]{MR3892404} plays a key r\^{o}le in the partial results on the planar version of Falconer's problem achieved in \cite{KeletiShmerkin19, GIOW20, Stull22}. However, the radial projection result in question only applies to planar sets of dimension $>1$. Theorem \ref{mainNew} (or rather its quantitative version, Corollary \ref{cor:radial}) can be seen as a substitute of \cite[Theorem 1.11]{MR3892404} for sets of dimension $\le 1$, and thus opens the door to improvements on the distance set problem in the critical case of dimension $1$. Similar considerations are valid in higher dimensions.

Restricting to cartesian products, Theorem \ref{mainNew} enables progress on another classical problem in geometric measure theory, the discretised sum-product problem. For example, the following new estimate of sum-product type follows immediately from Theorem \ref{mainNew} applied to $X=-A\times B$ and $Y= -B \times A$:
\begin{cor}
Let $A, B\subset\R$ be Borel sets. Then
\[
\Hd\left(\frac{A-B}{A-B}\right) \ge \min\{\Hd A + \Hd B, 1\}.
\]
\end{cor}

We hope to explore these connections further in future work.

\subsection{Paper outline} Theorem \ref{mainNew} and its more quantitative counterpart, Corollary \ref{cor:radial}, is proved in Section \ref{s:new}. Theorem \ref{main} is proved in Section \ref{deltaReduction}. More precisely, we start with Theorem \ref{main2}: this is a $\delta$-discretised statement vaguely reminiscent of Theorem \ref{main}, which can be deduced from the incidence theorem of Fu and Ren \cite{2021arXiv211105093F} with some effort. This version is still some distance away from proving Theorem \ref{main}. To bridge the gap, we resort to another "bootstrapping" argument, see Theorem \ref{mainThinTubes} and Lemma \ref{lemma7}.

In Section \ref{s:higherDim}, we apply the quantitative versions of Theorems \ref{mainNew} and \ref{main} in $\R^{2}$, combined with an integralgeometric tool \cite[Theorem 6.7]{2021arXiv211209044S}, to give new proofs of Dote and Gan's results, Theorems \ref{ganDote1}-\ref{ganDote2}, in general dimensions.

\subsection{Notation and some preliminaries} The notation $B(x,r)$ stands for a closed ball of radius $r > 0$ and centre $x \in X$, in a metric space $(X,d)$. In the case $X=\R^d$, we denote $B^d= B(0,1)$. If $A \subset X$ is a bounded set, and $r > 0$, we write $|A|_{r}$ for
the $r$-covering number of $A$, that is, the smallest number of
closed balls of radius $r$ required to cover $A$. Cardinality is
denoted $|A|$, and Lebesgue measure $\mathrm{Leb}(A)$. The closed
$r$-neighbourhood of $A$ is denoted $A^{(r)}$.

If $X,Y$ are positive numbers, then $X\lesssim Y$ means that $X\le CY$ for some constant $C$, while $X\gtrsim Y$, $X\sim Y$ stand for $Y\lesssim X$, $X\lesssim Y\lesssim X$. If the implicit constant $C$ depends on a parameter this will be mentioned explicitly or denoted by a subscript.

If $\mu$ is a positive finite measure on $\R^d$ and $t\ge 0$, the $t$-energy is defined as
\[
I_t(\mu) = \iint |x-y|^{-t}d\mu(x)d\mu(y) \in (0,\infty].
\]
It is well known that if $\mu$ satisfies a Frostman condition $\mu(B(x,r))\le C r^t$ for all balls $B(x,r)$ and some $C,t>0$, then $I_s(\mu)<\infty$ for all $s<t$. Conversely, if $I_t(\mu)<\infty$, then for every $\e>0$ there are a compact set $K$ with $\mu(\R^d\setminus K)<\e$ and a constant $C$ such that $\mu|_K(B(x,r))\le C r^t$ for all balls $B(x,r)$. Frostman's Lemma asserts that if $X\subset\R^d$ is a Borel set with $\Hd(X)>t\ge 0$, then there is a Borel probability measure $\mu$ with $\spt(\mu)\subset X$ and $I_t(\mu)<\infty$. See e.g. \cite[Chapter 8]{zbMATH01249699}. In the sequel we will use these facts without further reference.

\subsection*{Acknowledgements} Some ideas for this paper were conceived while the first two authors were visiting the Hausdorff Research Institute for Mathematics, Bonn, during the trimester \emph{Interactions between Geometric measure theory, Singular integrals, and PDE.} We would like to thank the institute and its staff for their generous hospitality. We also thank Katrin F\"assler, Jiayin Liu, and Josh Zahl for useful discussions. Part of this research was done while the third author was visiting Po Lam Yung at the Australian National University, she would like to thank him  and the analysis group for the warm hospitality.

\section{Proof of Theorem \ref{mainNew}}\label{s:new}

\subsection{Preliminaries}

In order to prove Theorem \ref{mainNew}, we need two black boxes. The first one is a weaker (although more quantitative) version of the statement itself, recorded in \cite[Theorem B.1]{Shmerkin20}. The second one is a recent $\epsilon$-improvement \cite{2021arXiv210603338O} in the \emph{Furstenberg set problem}.

We begin by recalling \cite[Theorem B.1]{Shmerkin20}. We write $\mathcal{P}(B)$ for the family of Borel probability measures on the metric space $B$ (usually $\R^{d}$ or the unit ball $B^{d}$ of $\R^{d}$). Throughout this section and the next, we use the following notation. If $X,Y \subset \R^{d}$ and $G \subset X \times Y$, we write
\begin{displaymath} G|_{x} := \{y \in Y : (x,y) \in G\} \quad \text{and} \quad G|^{y} := \{x \in X : (x,y) \in G\} \end{displaymath}
for all $x \in X$ and $y \in Y$. We note that if $G \subset \R^{d} \times \R^{d}$ is a Borel set, then $G|_{x}$ and $G|^{y}$ are also Borel for every $x,y \in \R^{d}$ (this is because $\mathrm{Bor}(\R^{d} \times \R^{d}) = \mathrm{Bor}(\R^{d}) \times \mathrm{Bor}(\R^{d})$).
		
	\begin{definition}[Thin tubes]
		Let $K,t \geq 0$ and $c \in (0,1]$. Let $\mu, \nu\in \mathcal{P}(\mathbb{R}^d)$ with $\spt(\mu) =: X$ and $\spt(\nu) =: Y$. We say that $(\mu, \nu)$ has \emph{$(t, K, c)$-thin tubes} if there exists a Borel set $G \subset X \times Y$ with $(\mu \times \nu)(G) \geq c$ with the following property. If $x \in X$, then
		\begin{equation}\label{form73} \nu(T\cap G|_{x}) \leq K \cdot r^{t} \quad \text{for all } r > 0 \text{ and all $r$-tubes $T$ containing $x$.} \end{equation}
		We also say that $(\mu,\nu)$ has $t$-thin tubes if $(\mu,\nu)$ has $(t,K,c)$-thin tubes for some $K,c > 0$.
			\end{definition}
			
Here, and below, an \emph{$r$-tube} is the $r$-neighbourhood of some line. The terminology of thin tubes was introduced in \cite{2021arXiv211209044S}, and there it was in fact called \emph{strong thin tubes}. We have no use for the notion of (weak) thin tubes, so we prefer the simpler terminology.

\begin{remark}\label{thinTubesAndDimension} Assume that $\mu,\nu \in \mathcal{P}(\R^{d})$ are such that the pair $(\mu,\nu)$ has $t$-thin tubes for some $t > 0$. Then there exists $x \in \spt(\mu)$ such that $\mathcal{H}^{t}(\pi_{x}(Y \, \setminus \, \{x\})) > 0$. To see this, pick $x \in X$ such that $\nu(G|_{x}) > 0$. Then also $\nu(G|_{x} \, \setminus \, \{x\}) > 0$, since otherwise \eqref{form73} is not possible with exponent $t > 0$. Therefore, we may pick a compact set $K \subset G|_{x} \, \setminus \, \{x\} \subset Y \, \setminus \, \{x\}$ with $\nu(K) > 0$. Now \eqref{form73} implies that the push-forward $\pi_{x}(\nu|_{K})$ satisfies a $t$-dimensional Frostman condition, and therefore $\mathcal{H}^{t}(\pi_{x}(K)) > 0$. \end{remark}

We are now prepared to state \cite[Theorem B.1]{Shmerkin20}, although we do so directly using the terminology of thin tubes (see \cite[Proposition 5.10]{2021arXiv211209044S}).

\begin{proposition}\label{ShWaProp} For every $C,\delta,\epsilon,s > 0$, there exist
\begin{displaymath} \beta := \beta(s) \in (0,s), \quad \tau := \tau(\epsilon,s) > 0 \quad \text{and} \quad K := K(C,\delta,\epsilon,s) > 0,\end{displaymath}
all depending continuously on the parameters, such that the following holds. Assume that $\mu,\nu \in \mathcal{P}(B^2)$  satisfy $\dist(\spt(\mu),\spt(\nu)) \geq C^{-1}$ and the $s$-dimensional Frostman condition $\mu(B(x,r)) \leq Cr^{s}$ and $\nu(B(x,r)) \leq Cr^{s}$ for all $x \in \R^{2}$ and $r > 0$. Assume further that $\nu(T) \leq \tau$ for all $\delta$-tubes $T \subset \R^{2}$.

Then $(\mu,\nu)$ has $(\beta,K,1 -\epsilon)$-thin tubes. \end{proposition}

Next, we discuss the discretised Furstenberg set estimate from \cite{2021arXiv210603338O}. It is essentially \cite[Theorem 1.3]{2021arXiv210603338O}, but again we already state it in a language suitable for our later application. We need to introduce the notion of \emph{$(\delta,s)$-sets}:

\begin{definition}[$(\delta,s,C)$-sets]\label{deltaSSet} Let $(X,d)$ be a metric space, let $P \subset X$ be a set, and let $\delta,C > 0$, and $s \geq 0$. We say that $P$ is a \emph{$(\delta,s,C)$-set} if
\begin{displaymath} |P \cap B(x,r)|_{\delta} \leq C r^{s} \cdot |P|_{\delta}, \qquad x \in X, \, r \geq \delta. \end{displaymath}
\end{definition}

The definition of a $(\delta,s,C)$-set, above, is slightly different from a more commonly used notion in the area, introduced by Katz and Tao \cite{MR1856956}, see Definition \ref{def:KatzTao} below. However, Definition \ref{deltaSSet} is not new either: these variants of $(\delta,s,C)$-sets were for example used in \cite{2021arXiv210603338O}. It is worth noting that a $(\delta,s,C)$-set is not required to be $\delta$-separated to begin with; however, it is easy to check that every $(\delta,s,C)$-set contains a $\delta$-separated $(\delta,s,C')$-set with $C' \sim C$. Another remark, to be employed without further mention, is that if $P \subset \R^{d}$ is a $(\delta,s,C)$-set, and $P' \subset P$ satisfies $|P'|_{\delta} \geq c|P|_{\delta}$, then $P'$ is a $(\delta,s,C/c)$-set.


\begin{definition}[$(\delta,s,C)$-sets of lines and tubes] Let $\mathcal{A}(2,1)$ be the metric space of all (affine) lines in $\R^{2}$ equipped the metric
\begin{displaymath} d_{\mathcal{A}(2,1)}(\ell_{1},\ell_{2}) := \|\pi_{L_{1}} - \pi_{L_{2}}\| + |a_{1} - a_{2}|. \end{displaymath}
Here $\pi_{L_{j}} \colon \R^{2} \to L_{j}$ is the orthogonal projection to the subspace $L_{j}$ parallel to $\ell_{j}$, and $\{a_{j}\} = L_{j}^{\perp} \cap \ell_{j}$. A set $\mathcal{L} \subset \mathcal{A}(2,1)$ is called a $(\delta,s,C)$-set if $\mathcal{L}$ is a $(\delta,s,C)$-set in the metric space $(\mathcal{A}(2,1),d_{\mathcal{A}(2,1)})$, in the sense of Definition \ref{deltaSSet}.

As before, a $\delta$-tube stands for the $\delta$-neighbourhood of some line $\ell \in \mathcal{A}(2,1)$. A family of $\delta$-tubes $\mathcal{T} = \{\ell(\delta) : \ell \in \mathcal{L}\}$ is called a $(\delta,s,C)$-set if the line family $\mathcal{L}$ is a $(\delta,s,C)$-set. Similarly, we define the $\delta$-covering number $|\mathcal{T}|_\delta$ to be $|\mathcal{L}|_\delta$.
\end{definition}

We can now state \cite[Theorem 1.3]{2021arXiv210603338O} in our language:
\begin{thm} \label{thm:improved-Furstenberg}
Given $s\in (0,1)$ and $t\in (s,2)$ there exists $\e=\e(s,t)>0$ such that the following holds for all $0<\delta\le \delta_0(s,t)$: if $X\subset B^2$ is an $(\delta,t,\delta^{-\e})$-set, and for each $x\in X$ there is an $(\delta,s,\delta^{-\e})$-set $\mathcal{T}'_x \subset\mathcal{T}^\delta$ of tubes passing through $x$, then
\[
\Big| \bigcup_{x\in X} \mathcal{T}'_x \Big|_\delta \ge \delta^{-2s-\e}.
\]
Moreover, $\e$ can be taken uniform in any compact subset of $\{ (s,t): s\in (0,1), t\in (s,2)\}$.
\end{thm}

To be precise, \cite[Theorem 1.3]{2021arXiv210603338O} is stated in terms of dyadic tubes and sets of dyadic squares, but it is straightforward to deduce Theorem \ref{thm:improved-Furstenberg} from it  (see also \cite[Theorem 3.1]{2021arXiv210603338O} for a similar deduction). The uniformity of $\e$ over compact sets is \cite[Remark 1.4]{2021arXiv210603338O}.

\subsection{The key lemma}

The next lemma contains the main work needed for Theorem \ref{mainNew}.

\begin{lemma}\label{lem: radial2}
	Let $0 < s \leq 1 $ and $\e \in (0,\tfrac{1}{10})$. Let $\mu,\nu\in\mathcal{P}(B^2)$ such that
\begin{align*}
\mu(B_r) &\le C\,r^s\quad (r>0),\\
\nu(B_r) &\le C\,r^s\quad (r>0),\\
\dist(\spt&(\mu),\spt(\nu))\ge C^{-1}
\end{align*}
for some constant $ C>0$.  Let $\beta=\beta(s)>0$ be the parameter from Proposition \ref{ShWaProp}. If both $(\mu, \nu)$  and $(\nu, \mu)$ have $(\sigma, K, 1 - \epsilon)$-thin tubes for some $\beta \leq \sigma < s$ then there exist $\eta = \eta(s,\sigma) > 0$ and  $K' = \max \{ K^{1/\eta}, 2^{-\frac{\log \eta \e -2}{\eta}},  K_0(C, \sigma, s)\} > 0$  such that $(\mu, \nu)$ and $ (\nu, \mu)$ have $(\sigma+\eta, K', 1 - 5\epsilon)$-thin tubes. Moreover, the value of $\eta(s,\sigma)$ is bounded away from zero on any compact subset of $\{(\sigma,s) \in (0,1]^{2} : \beta(s) \leq \sigma < s\}$.
\end{lemma}

\begin{proof}
	We argue by contradiction. Since the roles of $\mu, \nu$ are symmetric in the proof, we can assume  that $(\mu, \nu)$ do not have $(\sigma+\eta, K', 1-5\epsilon)$--strong thin tubes for $K' = \max \{ K^{1/\eta}, 2^{-\frac{\log \eta\e -2}{\eta}}, K_0\} > 0$, with $K_0=K_0(C,\sigma, s)\geq 1$ to be determined in the course of the proof. All the implicit constants below are allowed to depend on $C$. Let $X=\spt(\mu)$, $Y=\spt(\nu)$.

Since both $(\mu, \nu)$ and $(\nu,\mu)$ have $(\sigma, K, 1-\epsilon)$--thin tubes, there is a Borel set $G \subset X \times Y$ with $\mu\times\nu(G)>1-2\e$ such that
\begin{align}
\nu(T\cap G|_x) &\le K\cdot r^{\sigma} \text{ for all } r > 0 \text{ and  all } r\text{-tubes } T \text{ containing } x\in X, \label{eq: Y_{x,1}}\\
\mu(T\cap G|^y) &\le K\cdot r^{\sigma} \text{ for all } r > 0 \text{ and  all } r\text{-tubes } T \text{ containing } y\in Y. \label{eq: X_{y,1}}
\end{align}
For every $x \in X$ and $r \in 2^{-\N}$, let $\mathcal{T}_{x,r}''$ consist of those $r$-tubes which contain $x$ and satisfy
\begin{equation}\label{eq: lowerbd} \nu(T\cap G|_{x}) \geq K' \cdot r^{\sigma + \eta}. \end{equation}
The collection $\mathcal{T}_{x,r}''$ may be infinite, and this will cause inconvenience later. We fix this as follows. Let $\mathcal{T}^{r}$ be a family of $2r$-tubes such that
\begin{enumerate}[(i)]
\item \label{it:i} $|\mathcal{T}^{r}| \sim r^{-2}$,
\item \label{it:ii} if $T \subset \R^{2}$ is an $r$-tube, then $T \cap B^{2}$ is contained in at least one and at most $O(1)$ tubes from $\mathcal{T}^{r}$,
\item \label{it:iii} $\mathcal{T}^{r}$ is $cr$-separated for some universal $c>0$.
\end{enumerate}
The family $\mathcal{T}^{r}$, is easy to find, picking $\sim r^{-1}$ tubes in each direction from an $r$-net of directions. Note that a consequence of \eqref{it:iii} is that if $x,y \in B^{2}$, then
\begin{equation} \label{eq:separated-tubes}
|\{T \in \mathcal{T}^{r} : x,y \in T\}| \lesssim |x - y|^{-1}.
\end{equation}

Let $\mathcal{T}_{x,r}' \subset \mathcal{T}^{r}$ be a minimal set with the property that each intersection $T \cap B^{2}$ with $T \in \mathcal{T}_{x,r}''$ is contained in at least one element of $\mathcal{T}_{x,r}'$. Evidently \eqref{eq: lowerbd} remains valid for all $T \in \mathcal{T}_{x,r}'$, and moreover
\begin{equation}\label{form62} \left(\cup \mathcal{T}_{x,r}''\right) \cap B^{2}\subset \cup \mathcal{T}_{x,r}' \quad \text{and} \quad |\mathcal{T}_{x,r}'| \lesssim r^{-\sigma - \eta}. \end{equation}
We also note that since all the tubes in $\mathcal{T}_{x,r}' \subset \mathcal{T}^{r}$ contain $x \in \spt(\mu)$, they have bounded overlap on $\spt(\nu)$ (this follows from $\dist(\spt(\mu),\spt(\nu)) \geq C^{-1}$ and property \eqref{eq:separated-tubes} of $\mathcal{T}^{r}$).
We define
\[
\bar{H}_{r} := \{(x,y) \in X \times Y : y \in \cup \mathcal{T}_{x,r}'\}, \quad \text{and} \quad \bar{H} := \bigcup_{r} \bar{H}_{r}.
\]
We claim that
\begin{displaymath} (\mu \times \nu)(\bar{H}) \geq 3\epsilon. \end{displaymath}
Indeed, assume to the contrary that $(\mu \times \nu)(\bar{H}) < 3\epsilon$, thus $(\mu \times \nu)(\bar{H}^{c}) \geq 1 - 3\epsilon$, and
\begin{displaymath} (\mu \times \nu)(\bar{H}^{c} \cap G) > 1 - 5\epsilon. \end{displaymath}
Since $(\mu,\nu)$ do not have $(\sigma + \eta,K',1 - 5\epsilon)$-thin tubes by assumption, we infer that there exists a point $x \in X$, and an $r$-tube $T \subset \R^{2}$ containing $x$ which satisfies
\begin{displaymath}
\nu(T\cap G|_{x}) \geq \nu(T\cap (\bar{H}^{c} \cap G)|_{x}) \geq K' \cdot r^{\sigma + \eta}.
\end{displaymath}
While a priori this holds for some $r\in (0,1]$, we may take $r$ to be dyadic by enlarging $T$ and replacing $K'$ by $4 K'$. However, this means by definition (and the inclusion in \eqref{form62}) that $T \subset \bar{H}|_{x}$, so it is absurd that $\nu(T\cap (\bar{H}^{c})|_{x}) > 0$. This contradiction completes the proof of $(\mu \times \nu)(\bar{H}) \geq 3\epsilon$. We now set
\begin{displaymath} H := G \cap \bar{H} \quad \text{and} \quad H_{r} := G \cap \bar{H}_{r}, \end{displaymath}
and we infer from a combination of $(\mu \times \nu)(G) \geq 1 - 2\epsilon$ and $(\mu \times \nu)(\bar{H}) \geq 3\epsilon$ that
\begin{displaymath} (\mu \times \nu)(H) \geq \epsilon. \end{displaymath}

Let $r_1=r_1(C,s,\sigma)$ be a small dyadic parameter whose value will be determined along the rest of proof (we will often assume that it is small enough in terms of the given parameters without explicit mention). Let $r_2$ be such that $\sum_{r_2 \ge r\in 2^{-\N}} r^\eta < \e/2$ (for example, take $r_2= 2^{\frac{\log \eta\e-2}{\eta}}$).  Let
\[
r_0 =\min\{K^{-1/\eta}, r_2,  r_1(C,s,\sigma)\}.
\]
If $K'$ is taken so large that $K' \cdot r_0^{\sigma+\eta}>1$  (note that this matches the claimed form for $K'$), then $\mathcal{T}_{x,r}' =\emptyset$ for all $r > r_0$. This in particular implies that $H_{r}|_{x} \subset \cup \mathcal{T}_{x,r}'$ is empty for all $x \in X$ and $r > r_{0}$. In other words $H_{r} = \emptyset$ for $r > r_{0}$, which implies that $H$ is contained in the union of the sets $H_{r}$ with $r \leq r_{0}$. Since $(\mu \times \nu)(H) \geq \epsilon$, it now follows from the choice of $r_{1}\ge r_{0}$ that
\begin{displaymath} (\mu \times \nu)(H_{r}) \geq 2r^{\eta} \quad \text{for some} \quad r \leq r_{0}. \end{displaymath}
We fix this value of "$r$" for the remainder of the proof. We define
\begin{displaymath} \mathbf{X} := \{x \in X : \nu(H_{r}|_{x}) \geq r^{\eta}\}, \end{displaymath}
and note that $\mu(\mathbf{X}) \geq r^{\eta}$. Recall from the definition of $\bar{H}_{r} \supset H_{r}$ that the fibres $H_{r}|_x$ are covered by the tube families $\mathcal{T}_{x,r}'$, and by \eqref{form62} that $|\mathcal{T}_{x,r}'| \lesssim r^{ -\sigma -\eta}$.  For $x\in \mathbf{X}$, define
\[
\mathcal{T}_x := \{ T\in\mathcal{T}'_{x,r}: \nu(T\cap H_{r}|_{x}) \geq r^{\sigma+3\eta}\},\quad Y_x = (H_{r}|_{x}) \cap \bigcup \mathcal{T}_x.
\]
Since $\nu(H_{r}|_{x}) \geq r^{\eta}$ for $x \in \mathbf{X}$ and to form $Y_{x}$ we are removing $\lesssim r^{-\sigma-\eta}$ tubes of mass $< r^{\sigma+3\eta}$, we have $\nu(Y_{x}) \geq r^{2\eta}$ for all $x \in \mathbf{X}$. For every $x \in \mathbf{X}$, the set $\mathcal{T}_x$ of $r$-tubes covers $Y_x$ and satisfies
\begin{equation}\label{eq: eqfraction}
r^{\sigma + 3\eta} \leq  \nu(T\cap Y_x)   \leq r^{\sigma-\eta}, \qquad T\in \mathcal{T}_x,
\end{equation}
where the upper bound follows from $Y_{x} \subset H|_{x} \subset G|_{x}$, the choice of $r_0$, and \eqref{eq: Y_{x,1}}. In fact, more generally
\[
\nu(T^{(\rho)}\cap Y_x) \leq r^{-\eta}\rho^{\sigma}, \qquad \rho\in [r, 1], \, T\in \mathcal{T}_x.
\]

Putting these facts together, we see that
\begin{equation}\label{form61} r^{-\sigma + 3\eta} \leq r^{-\sigma + \eta} \cdot \nu(Y_{x})  \lesssim |\mathcal{T}_{x}| \leq |\mathcal{T}_{x,r}'| \lesssim r^{-\sigma - \eta}, \qquad x \in \mathbf{X}, \end{equation}
and
\[
|\{ T'\in\mathcal{T}_{x}: T'\subset T^{(\rho)} \}| \leq r^{-\sigma-4\eta}
\]
for and $T\in\mathcal{T}_{x}$. Since for the axial line $\ell$ of $\mathcal{T}$ we have
\[
B(\ell,\rho) \subset \{ \ell'\in \mathcal{A}(2,1): \ell'\cap B^2 \subset \ell^{(O(\rho))}\},
\]
we deduce that $\mathcal{T}_{x}$ is an $(r, \sigma, r^{-5\eta})$--set for each $x \in \mathbf{X}$. 

Let
\[
\kappa=\frac{14\eta}{s - \sigma}.
\]
(This choice will become clear at the end of the proof.) Call a tube $T\in\mathcal{T}_x$ \emph{concentrated} if there is a ball $B_T$ of radius $r^\kappa$ with
\[
\nu(T\cap B_T\cap Y_x)\ge \tfrac{1}{3} \cdot \nu(T\cap Y_x).
\]
Suppose first that there is $\mathbf{X}'\subset \mathbf{X}$ with $\mu(\mathbf{X}')\ge \mu(\mathbf{X})/2$ such that at least half of the tubes $\mathcal{T}_x$ are not concentrated for $x\in \mathbf{X}'$. Since $\mu(\mathbf{X}')\ge r^{\eta}/2$, we get from the mass distribution principle and the Frostman assumption on $\mu$ that $\mathcal{H}^\sigma_\infty(\mathbf{X}')\gtrsim C^{-1}r^{\eta}$. By the discrete form of Frostman's Lemma \cite[Proposition A.1]{FasslerOrponen14}, and if $r_{1} \geq r$ is sufficiently small in terms of $C$, there exists an $(r,\sigma,r^{-2\eta})$-set $P \subset \mathbf{X}'$.  Let $\mathcal{T}'_x$ be the subset of non-concentrated tubes for each $x\in P$; it is a $(r,\sigma, 2 r^{-5\eta})$--set of tubes. Let $\mathcal{T}' = \bigcup_{x\in P} \mathcal{T}'_x$.

Take $\eta\le \e(s, \sigma)^2$ for the function $\e(s, \sigma)$ in Theorem \ref{thm:improved-Furstenberg}. Note that the uniformity of $\e(s,\sigma)$ over compact subsets yields the corresponding uniformity of $\eta$ claimed in the statement. Then
\begin{equation} \label{eq:T-large}
|\mathcal{T}'| \gtrsim r^{-2\sigma -\sqrt{\eta}}.
\end{equation}
Counting cardinality makes sense, because each $\mathcal{T}_{x}' \subset \mathcal{T}_{x,r}' \subset \mathcal{T}^{r}$ was defined as subset of the common finite family $\mathcal{T}^{r}$, recall below \eqref{eq: lowerbd}. Moreover, by the separation property \eqref{it:iii} of $\mathcal{T}^{r}$, cardinality is comparable to $r$-covering number in this case.  

It follows from \eqref{eq: eqfraction} and the non-concentrated property that for each $T\in\mathcal{T}'$ there are two $r^\kappa$-separated sets $Y_{T,1}, Y_{T,2}\subset T$ with $\nu(Y_{T,j})\ge r^{\sigma+4\eta}$. Thus, recalling from property \eqref{eq:separated-tubes} of the family $\mathcal{T}^{r}$ that $|\{T \in \mathcal{T}^{r} : (y_{1},y_{2}) \in T\}| \lesssim |y_{1} - y_{2}|^{-1}$, we may infer that
\begin{align*} r^{2\sigma + 8\eta}|\mathcal{T}'| & \leq \sum_{T \in \mathcal{T}'} (\nu \times \nu)(Y_{T,1} \times Y_{T,2})\\
& = \iint \sum_{T \in \mathcal{T}'} \mathbf{1}_{Y_{T,1} \times Y_{T,2}}(y_{1},y_{2}) \, d(\nu \times \nu)(y_{1},y_{2})\\
& \lesssim \iint r^{-\kappa} \, d(\nu \times \nu)(y_{1},y_{2}) = r^{-\kappa}. \end{align*}
Thus $|\mathcal{T}'| \lesssim r^{-2\sigma - 8\eta - \kappa}$, which contradicts \eqref{eq:T-large} if first $\eta$ is taken small enough in terms of $s - \sigma$ and then $r_1$ is taken small enough in terms of $\eta$.

Assume next that there is a subset $\mathbf{X}' \subset \mathbf{X}$ with $\mu(\mathbf{X}') \geq \mu(\mathbf{X})/2$ such that at least half of the tubes in $\mathcal{T}_x$ are concentrated for all $x \in \mathbf{X}'$. Let $\mathcal{T}'_x$ denote the concentrated tubes, and let $\{ B_T: T\in\mathcal{T}'_x\}$ be the corresponding ``heavy'' $r^\kappa$-balls, as in the definition. Since the family $\mathcal{T}_x$ has bounded overlap on $\spt(\nu)$, the set
\[
H' = \{ (x,y):x\in \mathbf{X}', y\in T\cap B_T \cap Y_x \text{ for some } T\in\mathcal{T}'_x\} \subset H \subset G
\]
satisfies
\begin{align*}
(\mu\times\nu)(H') &\gtrsim \mu(\mathbf{X}') \cdot \inf_{x\in \mathbf{X}'} |\mathcal{T}_x| \cdot \inf_{x\in \mathbf{X}', T\in\mathcal{T}'_x} \nu(T\cap B_T\cap Y_x)\\
&\overset{\eqref{eq: eqfraction}-\eqref{form61}}{\gtrsim} r^{\eta}\cdot r^{-\sigma + 3\eta} \cdot r^{\sigma+3\eta} = r^{7\eta}.
\end{align*}
Note that if $(x,y)\in H'$, then there is a tube $T(x,y) \in \mathcal{T}^{r}$ containing $x,y$ such that
\[
\nu\big(B(y,2r^\kappa)\cap T(x,y)\big) \gtrsim r^{\sigma+3\eta}.
\]
Further, this $\nu$ measure cannot be too concentrated near $y$, because
\[
\nu(B(y,r)) \leq C \cdot r^{s} \leq \tfrac{1}{2} \cdot \nu(B(y,2r^{\kappa} \cap T(x,y)),
\]
assuming that $3\eta < s - \sigma$ and $r_{1} \geq r$ is small enough in terms of $C$ and $\eta$. Therefore, for each $(x,y)\in H'$ we can pigeonhole a dyadic number $r \leq \xi(x,y) \leq r^\kappa$ such that
\[
\nu\big(A(y,\xi(x,y),2\xi(x,y))\cap T(x,y)\big) \geq r^{\sigma+4\eta},
\]
where $A(y,\xi,2\xi) = \{x \in \R^{2} : \xi \le |x-y| <2\xi\}$. Then, recalling that $(\mu \times \nu)(H') \gtrsim r^{7\eta}$, we can further pigeonhole a value $r \le \xi\leq r^\kappa$ such that
\[
(\mu\times\nu)(H'') \geq r^{8\eta},\quad\text{where} \quad H''=\{ (x,y)\in H' : \xi(x,y)=\xi\} \subset G.
\]
Fix $y \in Y$ such that $\mu(H''|^y) \geq r^{8\eta}$ for the rest of the proof. Observe that $H''|^{y}$ can be covered by a collection of tubes $\mathcal{T}_{y} \subset \mathcal{T}^{r}$ which contain $y$, and satisfy
\begin{equation}\label{form63} \nu(A(y,\xi,2\xi) \cap T) \geq r^{\sigma + 4\eta}, \qquad T \in \mathcal{T}_{y}. \end{equation}
We claim that $\mathcal{T}_{y}$ contains $\gtrsim r^{9\eta} \cdot (\xi/r)^{\sigma}$ elements whose directions are separated by $\geq (r/\xi)$. Indeed, if $\mathbf{T}$ is any $(r/\xi)$-tube containing $y$, then
\[
\mu(\mathbf{T} \cap H''|^{y}) \leq \mu(\mathbf{T} \cap G|^y) \stackrel{\eqref{eq: X_{y,1}}}{\le} K \cdot (r/\xi)^{\sigma} \leq r^{-\eta} \cdot (r/\xi)^\sigma.
\]
Thus, it takes $\gtrsim \mu(H''|^{y}) \cdot r^{\eta} \cdot (\xi/r)^{\sigma} \gtrsim r^{9\eta} \cdot (\xi/r)^{\sigma}$ tubes of width $(r/\xi)$ to cover $H''|^{y}$, and perhaps even more $(r/\xi)$-tubes to cover the union $\cup \mathcal{T}_{y}$.  Let $\mathcal{T}_{y}' \subset \mathcal{T}_{y}$ be a maximal subset with $(r/\xi)$-separated directions. Thus $|\mathcal{T}_{y}'| \gtrsim r^{9\eta} \cdot (\xi/r)^{\sigma}$. This proves the claim.

The usefulness of the previous claim stems from the simple geometric fact that the family $\mathcal{T}^{y}$ has bounded overlap in $\R^{2} \, \setminus \, B(y,\xi)$. Therefore, we may infer from a combination of \eqref{form63} and the Frostman condition on $\nu$ that
\begin{displaymath} r^{13\eta} \cdot \xi^{\sigma} \lesssim r^{\sigma + 4\eta} \cdot |\mathcal{T}_{y}'| \lesssim \nu(B(y,2\xi)) \leq C \cdot (2\xi)^{s}, \end{displaymath}
or in other words $r^{13\eta} \lesssim \xi^{s - \sigma} \leq   r^{\kappa(s - \sigma)}$.
This is a contradiction by the choice $\kappa=14\eta/(s - \sigma)$, if $r_0$ is taken small enough in terms of $C,\eta,s,\sigma$.
\end{proof}

\subsection{Conclusion of the proof}

Lemma \ref{lem: radial2} will be used (multiple times) to infer the following corollary:

\begin{cor}\label{cor:radial} For all $0 < \sigma < s \leq 1$ and $C,\epsilon,\delta > 0$, there exist $\tau = \tau(\epsilon,\sigma,s) > 0$ and $K = K(C,\delta,\epsilon,s,\sigma) > 0$ such that the following holds. Assume that $\mu,\nu \in \mathcal{P}(B^{2})$ satisfy $\mu(B(x,r)) \leq Cr^{s}$, $\nu(B(x,r)) \leq Cr^{s}$, $\dist(\spt(\mu),\spt(\nu)) \geq C^{-1}$, and
\begin{equation}\label{form71} \max\{\mu(T),\nu(T)\} \leq \tau \end{equation}
for all $\delta$-tubes $T \subset \R^{2}$.

Then, both $(\mu,\nu)$ and $(\nu,\mu)$ have $(\sigma,K,1 - \epsilon)$-thin tubes.
\end{cor}

\begin{proof} It is immediate from Proposition \ref{ShWaProp} that both $(\mu,\nu)$ and $(\nu,\mu)$ have $(\beta,K,1 - \epsilon)$-thin tubes for some $\beta = \beta(s) \in (0,s)$ and $K = K(C,\delta,\epsilon,s) > 0$. If $\beta \geq \sigma$, the proof ends here. Otherwise $0 < \beta < \sigma$, and our task is to upgrade $\beta$ to $\sigma$. This will be done by iterating Lemma \ref{lem: radial2}.

We turn to the details. The first point is to be careful with "$\epsilon$". Indeed, let us fix $\bar{\epsilon} = \bar{\epsilon}(\epsilon,s,\sigma) > 0$ to be determined later, see \eqref{form59}. Then, instead of applying Proposition \ref{ShWaProp} directly with the parameter "$\epsilon$" given in Corollary \ref{cor:radial}, we apply Proposition \ref{ShWaProp} with the parameter "$\bar{\epsilon}$". The conclusion is the same as before: $(\mu,\nu)$ and $(\nu,\mu)$ have $(\beta,K_{1},1 - \bar{\epsilon})$-thin tubes for some $K_{1} = K_{1}(C,\delta,\bar{\epsilon},\sigma,s) > 0$, provided that
\begin{displaymath} \max\{\mu(T),\nu(T)\} \leq \tau(\bar{\epsilon},s) = \tau(\epsilon,\sigma,s) \end{displaymath}
for all $\delta$-tubes $T \subset \R^{2}$.

Now, if $\bar{\epsilon} > 0$ is sufficiently small, Lemma \ref{lem: radial2} shows that both $(\mu,\nu)$ and $(\nu,\mu)$ have $(\beta + \eta,K_{2},1 - 5\bar{\epsilon})$-thin tubes for some
\begin{displaymath} \eta = \eta(\beta,s,\sigma) = \eta(s,\sigma) > 0. \end{displaymath}
If $\beta + \eta \geq \sigma$, we are done. Otherwise, if $5\bar{\epsilon}$ remains sufficiently small, Lemma \ref{lem: radial2} says that $(\mu,\nu)$ and $(\nu,\mu)$ have $(\beta + 2\eta,K_{3},1 - 25\bar{\epsilon})$-thin tubes.

Continuing in this manner for $N \sim \eta^{-1} \sim_{s,\sigma} 1$ steps, Lemma \ref{lem: radial2} will bring us to a point where $(\mu,\nu)$ and $(\nu,\mu)$ have $(\sigma,K_{N},1 - 5^{N}\bar{\epsilon})$-thin tubes. Now, we choose $\bar{\epsilon} = \bar{\epsilon}(s,\sigma,\epsilon) > 0$ to be initially so small that
\begin{equation}\label{form59} 5^{N}\bar{\epsilon} \leq \epsilon. \end{equation}
With this choice, we have shown that $(\mu,\nu)$ and $(\nu,\mu)$ have $(\sigma,K,1 - \epsilon)$-thin tubes with constant $K = K_{N} > 0$. This completes the proof. \end{proof}

 As a corollary of the corollary, we record the following statement which is less quantitative, but more pleasant to use:
 \begin{cor}\label{cor:radial2} Let $s \in (0,1]$, and let $\mu,\nu \in \mathcal{P}(\R^{2})$ be measures which satisfy the $s$-dimensional Frostman condition $\mu(B(x,r)) \lesssim r^{s}$ and $\nu(B(x,r)) \lesssim r^{s}$ for all $x \in \R^{2}$ and $r > 0$. Assume that $\mu(\ell)\nu(\ell) < 1$ for every line $\ell \subset \R^{2}$. Then $(\mu,\nu)$ has $\sigma$-thin tubes for all $0 \leq \sigma < s$.  \end{cor}

 \begin{proof} Assume first that $\nu(\ell) > 0$ for some line $\ell \subset \R^{2}$. Then either $\mu(\ell) = 1$ or $\mu(\ell) < 1$. The second case is easy: then $\mu(\R^{2} \, \setminus \, \ell) > 0$, so there exists a compact set $K \subset \R^{2} \, \setminus \, \ell$ with $\mu(K) > 0$. Now, it is easy to check that $(\mu|_{K},\nu|_{\ell})$ has $s$-thin tubes. This implies that $(\mu,\nu)$ has $s$-thin tubes. Assume then that $\mu(\ell) = 1$. Then $\nu(\ell) < 1$ by assumption, so there exists a compact set $K \subset \R^{2} \, \setminus \, \ell$ such that $\nu(K) > 0$. Now it follows from Lemma \ref{l:kaufman} below that $(\mu|_{\ell},\nu|_{K})$ (and hence $(\mu,\nu)$) has $\sigma$-thin tubes for all $0 \leq \sigma < s$.

We may assume in the sequel that $\nu(\ell) = 0$ for all lines $\ell \subset \R^{2}$. Assume next that $\mu(\ell) > 0$ for some line $\ell \subset \R^{2}$. Since $\nu(\ell) = 0$, there exists a compact set $K \subset \R^{2} \, \setminus \, \ell$ such that $\nu(K) > 0$. Then $(\mu|_{\ell},\nu|_{K})$ has $\sigma$-thin tubes for all $0 \leq \sigma < s$ by Lemma \ref{l:kaufman} below, which formally implies that $(\mu,\nu)$ has $\sigma$-thin tubes for all $0 \leq \sigma < s$.

 Assume finally that $\mu(\ell) = 0 = \nu(\ell)$ for all lines $\ell \subset \R^{2}$. Pick two disjoint closed balls $B_{\mu},B_{\nu}$ such that $\mu(B_{\mu}) > 0$ and $\nu(B_{\nu}) > 0$, and restrict (and renormalise) $\mu$ and $\nu$ to $B_{\mu}$ and $B_{\nu}$, respectively. After this, we may assume that $\mu,\nu$ have disjoint supports. Fix $\sigma < s$, and let $\tau = \tau(\tfrac{1}{2},\sigma,s) > 0$ be the parameter given by Corollary \ref{cor:radial}. By \cite[Lemma 2.1]{MR3892404}, there exists $\delta > 0$ such that $\max\{\mu(T),\nu(T)\} \leq \tau$ for all $\delta$-tubes $T \subset \R^{2}$. Now, it follows from Corollary \ref{cor:radial} that $(\mu,\nu)$ have $(\sigma,K,\tfrac{1}{2})$-thin tubes for some $K > 0$, which depends on $\delta,s,\sigma$, the Frostman constants of $\mu,\nu$, and the distance between their supports. In particular, $(\mu,\nu)$ has $\sigma$-thin tubes. This concludes the proof. \end{proof}

In the proof above, we needed the following lemma:
\begin{lemma}\label{l:kaufman} Let $s \in (0,1]$, and let $\mu,\nu \in \mathcal{P}(\R^{2})$ be measures with separated supports satisfying the $s$-dimensional Frostman condition $\mu(B(x,r)) \lesssim r^{s}$ and $\nu(B(x,r)) \lesssim r^{s}$ for all $x \in \R^{2}$ and $r > 0$. Assume, moreover, that there exists a line $\ell \subset \R^{2}$ such that $\spt(\mu) \subset \ell$ and $\spt(\nu) \subset \R^{2} \, \setminus \, \ell$. Then $(\mu,\nu)$ have $\sigma$-thin tubes for all $0 \leq \sigma < s$.
\end{lemma}

\begin{proof} Fix $0 \leq \sigma < s$. Then $I_{\sigma}(\nu) < \infty$, and the standard proof of \emph{Kaufman's projection theorem}, see \cite[p. 56]{MR3617376}, shows that
\begin{equation}\label{form76} \int_{\ell} I_{\sigma}(\pi_{x}\nu) \, d\mu(x) \lesssim I_{\sigma}(\nu) < \infty. \end{equation}
Indeed, the only estimate needed to prove this inequality is
\begin{displaymath} \mu(\{x \in \ell : |\pi_{x}(y_{1}) - \pi_{x}(y_{2})| \leq \delta\}) \lesssim \left(\frac{\delta}{|y_{1} - y_{2}|}\right)^{s}, \quad y_{1},y_{2} \in \spt(\nu), \, y_{1} \neq y_{2}, \end{displaymath}
for all $\delta > 0$ small enough, and this is easy to verify by hand under the assumptions of the lemma. Alternatively, \eqref{form76} follows by applying a projective transformation that sends $\ell$ to the line at infinity, under which the radial projections $\pi_x$, $x\in\ell$ become orthogonal projections, and then one can literally apply the calculation in \cite[p. 56]{MR3617376} before undoing the projective transformation.


 From \eqref{form76} it follows that for $\mu$ almost every $x \in \ell$, the measure $\pi_{x}\nu$ restricted to a subset of positive measure satisfies a $\sigma$-dimensional Frostman condition. This implies that $(\mu,\nu)$ has $\sigma$-thin tubes, as claimed. \end{proof}

We are then prepared to prove Theorem \ref{mainNew}, whose statement we repeat here:

\begin{thm} Let $X,Y \subset \R^{2}$ be non-empty Borel sets, where $X$ is not contained on any line. Then, $\sup_{x \in X} \Hd \pi_{x}(Y \, \setminus \, \{x\}) \geq \min\{\Hd X,\Hd Y,1\}$. \end{thm}

\begin{proof} Abbreviate $t := \min\{\Hd X,\Hd Y,1\}$. We may assume that $t > 0$, since otherwise there is nothing to prove. We start by disposing of a special case where
\begin{equation}
\label{form77} \sup_{\ell\in\mathcal{A}(2,1)} \Hd(Y \cap \ell) \geq t.
\end{equation}
If the above holds, then for every $\epsilon > 0$, there exists a line $\ell_{\epsilon} \subset \R^{2}$ such that $\Hd(Y \cap \ell_{\epsilon}) \geq t - \epsilon$. Since we assumed that $X$ does not lie on a line, we may choose points $x_{\epsilon} \in X \, \setminus \, \ell_{\epsilon}$, and then
\begin{displaymath} \sup_{x \in X} \Hd \pi_{x}(Y \, \setminus \, \{x\}) \geq \sup_{\epsilon > 0} \Hd \pi_{x_{\epsilon}}(Y \cap \ell_{\epsilon}) = \sup_{\epsilon > 0} \Hd (Y \cap \ell_{\epsilon}) \geq t. \end{displaymath}
Next, we assume the opposite of \eqref{form77}: there exists $\epsilon_{0} > 0$ such that
\begin{equation}\label{form58} \Hd (Y \cap \ell) \leq t - \epsilon_{0}, \quad \ell\in\mathcal{A}(2,1).\end{equation}
Fix $t - \epsilon_{0} < s < t \leq \min\{\Hd X,\Hd Y\}$, and let $\mu,\nu \in \mathcal{P}(\R^{2})$ with $\spt(\mu) \subset X$ and $\spt(\nu) \subset Y$ such that $\mu(B(x,r)) \leq Cr^{s}$ and $\nu(B(x,r)) \leq Cr^{s}$ for all $x \in \R^{2}$ and $r > 0$. Assume, first, that $\mu(\ell) = 1$ for some line $\ell \subset \R^{2}$. In this case, we infer from \eqref{form58} that $\nu(\ell) = 0$ for this particular line $\ell$, and therefore $\nu(K) > 0$ for some compact set $K \subset \R^{2} \, \setminus \, \ell$. Now it follows from Lemma \ref{l:kaufman} that $(\mu|_{\ell},\nu|_{K})$ has $\sigma$-thin tubes for all $\sigma < s$, and in particular $\sup_{x \in X} \Hd \pi_{x}(Y \, \setminus \, \{x\}) \geq s$.

Finally, assume that $\mu(\ell) < 1$ for all lines $\ell \subset \R^{2}$. In this case it follows directly from Corollary \ref{cor:radial2} that $(\mu,\nu)$ has $\sigma$-thin tubes for all $\sigma < s$, so $\sup_{x \in X} \Hd \pi_{x}(Y \, \setminus \, \{x\}) \geq s$. Since $s < t$ was arbitrary, this completes the proof. \end{proof}

\section{Proof of Theorem \ref{main}}\label{deltaReduction}

\subsection{A discretised version of Theorem \ref{main}}

The purpose of this section is to prove Theorem \ref{main}. The main work consists of establishing a $\delta$-discretised version, stated in Theorem \ref{main2}. This theorem will discuss $(\delta,s,C)$-sets of $\delta$-tubes $\mathcal{T}_{x}$, $x \in B^2$, with the special property that $x \in T$ for all $T \in \mathcal{T}_{x}$. In this case, it is easy to check that the $(\delta,s,C)$-set property of $\mathcal{T}_{x}$ is equivalent to the statement that the directions of the tubes (as a subset of $S^{1}$) form a $(\delta,s,C')$-set for some $C' \sim C$.

\begin{thm}\label{main2}
For every $t \in (1,2]$, $\sigma \in [0,1)$, and $\zeta > 0$,
there exist $\epsilon = \epsilon(\sigma,t,\zeta) > 0$ and $\delta_{0} = \delta_{0}(\sigma,t,\zeta) > 0$
such that the following holds for all $\delta \in (0,\delta_{0}]$.

Let $s \in [0,2]$. Let $P_{K} \subset B^{2}$ be a $\delta$-separated $(\delta,t,\delta^{-\epsilon})$-set, and let $P_{E} \subset B^{2}$ be a $\delta$-separated $(\delta,s,\delta^{-\epsilon})$-set. Assume
that for every $x \in P_{E}$, there exists a
$(\delta,\sigma,\delta^{-\epsilon})$-set of tubes
$\mathcal{T}_{x}$ with the properties $x \in T$ for all $T \in
\mathcal{T}_{x}$, and
\begin{displaymath} |T \cap P_{K}| \geq \delta^{\sigma + \epsilon}|P_{K}|, \qquad T \in \mathcal{T}_{x}. \end{displaymath}
Then $\sigma \geq s + t - 1 - \zeta$.
\end{thm}

\begin{remark} It is easy to decipher from the proof the value of $\epsilon(\sigma,t,\zeta) > 0$ is bounded away from zero on compact subsets of $[0,1) \times (1,2] \times (0,1]$. \end{remark}

\begin{remark} The lower bound $\sigma \geq s + t - 1 - \zeta$ may appear odd if $s + t > 2$. In this case the lemma simply says that the hypotheses cannot hold for any $\sigma \in [0,1)$ (and for $\delta,\epsilon > 0$ sufficiently small). This is consistent with the fact that if $\mu,\nu$ are disjointly supported Frostman probability measures with exponents $s \in [0,2]$ and $t \in (1,2]$, respectively, and $s + t > 2$, then $\pi_{x}(\nu) \ll \mathcal{H}^{1}|_{S^{d - 1}}$ for $\mu$ almost every $x \in \R^{2}$, see \cite[Theorem 1.11]{MR3892404}. \end{remark}

Theorem \ref{main2} will be derived from a recent incidence theorem of
Fu and Ren \cite[Theorem 4.8]{2021arXiv211105093F} concerning $(\delta,s)$-sets of points and $(\delta,t)$-sets of tubes. The theorem of Fu and Ren is formulated in terms of a slightly different (and more classical) notion of $(\delta,s)$-sets. We start by stating this definition, and then we explore the connection to (our) $(\delta,s)$-sets.

\begin{definition}[Katz-Tao $(\delta,s)$-set]\label{def:KatzTao} Let $(X,d)$ be a metric space. We say that a $\delta$-separated set $P \subset X$ is a \emph{Katz-Tao $(\delta,s,C)$-set} if
\begin{displaymath} |P \cap B(x,r)| \leq C\left(\frac{r}{\delta} \right)^{s}, \qquad x \in \R^{d}, \, r \geq \delta. \end{displaymath}
\end{definition}

As the name suggests, the Katz-Tao $(\delta,s)$-sets were introduced by Katz and Tao \cite{MR1856956}. The next lemma shows that $(\delta,t)$-sets can be decomposed into Katz-Tao $(\delta,t)$-sets:

\begin{lemma}\label{lemma3} Let $(X,d)$ be a doubling metric space with constant $D \geq 1$.\footnote{Every ball of radius $r > 0$ can be covered by $D \in \N$ balls of radius $r/2$, with $D$ independent of $r$.} For every $\epsilon,t > 0$, there exists $\delta_{0} = \delta_{0}(\epsilon,D,t) > 0$ such that the following holds for all $\delta \in (0,\delta_{0}]$. Let $P \subset B(x_{0},1) \subset X$ be a $\delta$-separated $(\delta,t,C)$-set. Then $P$ can be written as a disjoint union
\begin{displaymath} P = P_{1} \cup \ldots \cup P_{N} \end{displaymath}
where each $P_{j}$ is a Katz-Tao $(\delta,t,1)$-set, and $N \leq C|P|\delta^{t-\epsilon}$.
\end{lemma}

\begin{proof} The method is the same as in the proof of \cite[Proposition 4.5]{2021arXiv211105093F}. For each dyadic rational $r \in [\delta,2]$, let $\mathcal{B}_{r}$ be a cover of $B(x_{0},1)$ by balls of radius $r$ with overlap $\lesssim_{D} 1$, and with the property that every ball of radius $\leq r/2$ intersecting $B(x_0,1)$ is contained in at least one of the balls in $\mathcal{B}_{r}$. To construct $\mathcal{B}_{r}$, choose a maximal $(r/2)$-separated set $X_{r} \subset B(x_0,1)$, and set $\mathcal{B}_{r} := \{B(x,r) : x \in X_{r}\}$. The bounded overlap of $\mathcal{B}_{r}$ follows from the doubling hypothesis of $X$. Indeed, if some $x \in B(x_0,1)$ lies in $N$ distinct balls of $\mathcal{B}_{r}$, then the centres of these balls form an $(r/2)$-separated subset of $B(x,2r)$ of cardinality $N$. By the doubling hypothesis, $N \lesssim_{D} 1$.

We then begin the proof in earnest. Assume that $P \neq \emptyset$, otherwise there is nothing to prove. Set
\begin{equation}\label{defH} H := 4^{t+1} C|P|\delta^{t} \geq 1. \end{equation}
The lower bound "$\geq 1$" follows from the assumption that $P \neq \emptyset$ is a $(\delta,t,C)$-set. Fix $r \in 2^{-\N}$ with $r \in [\delta,1]$ and $B \in \mathcal{B}_{r}$. We divide the points in $P \cap B$ into $m(B) = \ceil{|P \cap B|/H}$ groups $G^{r}_{1}(B),\ldots,G^{r}_{m(B)}(B)$ by forming as many groups of size exactly "$H$" as possible, and then one remainder group of size $\leq H$. It is of course possible that $|P \cap B| < H$: in this case $m(B) = 1$, and we only have the remainder group.

We then form a graph $G = (P,E)$, whose edge set $E \subset P \times P$ is defined as follows. For every group $G_{j}^{r}(B)$, $1 \leq j \leq m(B)$, we connect all the members of the group to each other by an edge. Then, we do this for every $B \in \mathcal{B}_{r}$, and for every dyadic rational $r \in [\delta,1]$.

What is the maximum degree of $G$? For every $x \in P$ and $r \in [\delta,1]$ fixed, $x$ is connected to every other point in its own group $G^{r}_{j}(B)$. This statement holds whenever $B \in \mathcal{B}_{r}$ contains $x$, and this may happen for $\lesssim_{D} 1$ different choices $B \in \mathcal{B}_{r}$. So, every $r \in [\delta,1]$ yields $\lesssim_{D} H$ edges incident to $x$. The number of dyadic scales $r \in [\delta,1]$ is $\sim \log (1/\delta)$, so $\max_{x \in P} \deg_{G}(x) \lesssim_{D} H\log (1/\delta)$.

As in \cite[Lemma 4.7]{2021arXiv211105093F}, we may now deduce from Brook's Theorem (see \cite{MR12236} or \cite{MR396344}) that the graph $G$ admits a colouring of the vertices $P$ with $N \lesssim_{D} H\log(1/\delta)$ colours with the property that no two adjacent vertices share the same colour. The colouring induces a partition $P = P_{1} \cup \ldots \cup P_{N}$, where
\begin{displaymath} N \lesssim_{D} H\log(1/\delta) \stackrel{\eqref{defH}}{=} 4^{t+1} C|P|\delta^{t} \log(1/\delta). \end{displaymath}
In particular, $N \leq C|P|\delta^{t - \epsilon}$ if $\delta > 0$ is small enough, depending only on $\epsilon,D,t$.

It remains to check that each $P_{j}$ is a Katz-Tao $(\delta,t,1)$-set. Fix $1 \leq j \leq N$. Let $x\in X$ and $\delta \leq \rho \leq 1$. If $B(x,\rho) \cap B(x_{0},1) = \emptyset$, then $|P_{j} \cap B(x,\rho)| = 0$. Otherwise $B(x,\rho)$ is contained in one of the balls $B \in \mathcal{B}_{r}$ for some $r/4 \leq \rho \leq r/2$. By the $(\delta,t,C)$-property of $P$, we have
\begin{displaymath} |P \cap B| \leq Cr^{t}|P| \leq 4^{t}(C\rho^{t}|P|). \end{displaymath}
This implies that the number "$m(B)$" of groups $G_{1}^{r}(B),\ldots,G_{m(B)}^{r}(B)$ is at most
\begin{displaymath} m(B) = \ceil{|P \cap B|/H} \leq \ceil{4^{t}4^{-t-1}(\rho/\delta)^{t}} \leq (\rho/\delta)^{t}. \end{displaymath}
 The set $P_{j}$ contains at most one point in each of these groups, therefore $|P_{j} \cap B(x,\rho)| \leq |P_{j} \cap B| \leq (\rho/\delta)^{t}$. This completes the proof of the lemma. \end{proof}




We next state the incidence theorem of Fu and Ren \cite[Theorem 4.8]{2021arXiv211105093F}. Given a set of points $P$ and a set of tubes $\mathcal{T}$,
we let $\mathcal{I}(P,\mathcal{T}) := \{(p,T) : p \in T\}$ be the set of \emph{incidences} between $P$ and $\mathcal{T}$.
 The original version was stated for Katz-Tao $(\delta,s)$-sets, but the version below will follow by combining the original statement with Lemma \ref{lemma3}:
\begin{thm}\label{FuRen} Let $0 \leq s,t \leq 2$. Then, for every $\epsilon > 0$, there exist $\delta_{0} = \delta_{0}(\epsilon) > 0$ such that the following holds for all $\delta \in (0,\delta_{0}]$. If $P \subset B^{2}$ is a  $\delta$-separated $(\delta,s,\delta^{-\epsilon})$-set, and $\mathcal{T}$ is a $\delta$-separated $(\delta,t,\delta^{-\epsilon})$-set, then
\begin{displaymath} |\mathcal{I}(P,\mathcal{T})| \leq |P||\mathcal{T}| \cdot \delta^{\kappa(s + t - 1) - 5\epsilon}, \end{displaymath}
where $\kappa = \kappa(s,t) = \min\{1/2,1/(s + t - 1)\}$. \end{thm}

\begin{proof}[Proof of Theorem \ref{FuRen}] By Lemma \ref{lemma3} applied in both $\R^{2}$ and $\mathcal{A}(2,1)$, we may write
\begin{displaymath} P = P_{1} \cup \ldots \cup P_{M} \quad \text{and} \quad \mathcal{T} = \mathcal{T}_{1} \cup \ldots \cup \mathcal{T}_{N}, \end{displaymath}
where $M \leq |P|\delta^{s - 2\epsilon}$ and $N \leq |\mathcal{T}|\delta^{t - 2\epsilon}$, each $P_{j}$ is a Katz-Tao $(\delta,s,1)$-set, and each $\mathcal{T}_{j}$ is a Katz-Tao $(\delta,t,1)$-set.
 By the original version of \cite[Theorem 4.8]{2021arXiv211105093F}, we have
\begin{displaymath} |\mathcal{I}(P_{i},\mathcal{T}_{j})| \leq \delta^{-s - t} \cdot \delta^{\kappa(s + t - 1) - \epsilon}, \qquad 1 \leq i \leq M, \, 1 \leq j \leq N, \end{displaymath}
assuming that $\delta = \delta(\epsilon) > 0$ is small enough. Therefore,
\begin{displaymath} |\mathcal{I}(P,\mathcal{T})| \leq \sum_{i = 1}^{M} \sum_{j = 1}^{N} |\mathcal{I}(P_{i},\mathcal{T}_{j})| \leq MN \cdot \delta^{-s - t} \cdot \delta^{\kappa(s + t - 1) - \epsilon} \leq |P||\mathcal{T}| \cdot \delta^{\kappa(s + t - 1) - 5\epsilon}. \end{displaymath}
This concludes the proof. \end{proof}

The next lemma allows us to find $(\delta,s)$-sets inside $\delta$-discretised Furstenberg sets.

\begin{lemma}\label{lemma4} For every $\xi > 0$, there exists $\delta_{0} = \delta_{0}(\xi) > 0$ and $\epsilon = \epsilon(\xi) > 0$ such that the following holds for all $\delta \in (0,\delta_{0}]$. Let $s \in [0,1]$ and $t \in [0,2]$. Assume that $\mathcal{T}$ is a non-empty $(\delta,t,\delta^{-\epsilon})$-set of $\delta$-tubes in $\R^{2}$. Assume that for every $T \in \mathcal{T}$ there exists a non-empty $(\delta,s,\delta^{-\epsilon})$-set $P_{T} \subset T \cap B^{2}$.
Then, the union
\begin{equation}\label{form32a} P := \bigcup_{T \in \mathcal{T}} P_{T} \end{equation}
contains a non-empty $(\delta,\gamma(s,t),\delta^{-\xi})$-set, where
\begin{equation}\label{form31} \gamma(s,t) = s + \min\{s,t\}. \end{equation}
\end{lemma}

\begin{remark} The constants $\delta_{0},\epsilon > 0$ can indeed be taken independent of $s,t$, the chief reason being that the lemma also holds with $s \in \{0,1\}$ and $t \in \{0,2\}$. \end{remark}


\begin{proof}[Proof of Lemma \ref{lemma4}] We only sketch the argument, since it is nearly follows from existing statements, and the full details are very standard (if somewhat lengthy). The main point is the following: it is known that the Hausdorff dimension of every $(s,t)$-Furstenberg set $F \subset \R^{2}$ satisfies $\Hd F \geq \gamma(s,t)$, where $\gamma(s,t)$ is the function defined in \eqref{form31}. The case $t \leq s$ is due to Lutz and Stull \cite{MR4179019}; they used information theoretic methods, but a more classical proof is also available, see \cite[Theorem A.1]{HSY21}. The case $t \geq s$ essentially goes back to Wolff in \cite{Wolff99}, but also literally follows from \cite[Theorem A.1]{HSY21}.

While the statement in \cite[Theorem A.1]{HSY21} only concerns Hausdorff dimension, the proof goes via Hausdorff content,
 and the following statement can be extracted from the argument. Let $P$ be the set defined in \eqref{form32a}. Then, the $\gamma(s,t)$-dimensional Hausdorff content of the $\delta$-neighbourhood $P(\delta)$ satisfies
\begin{equation}\label{form33} \mathcal{H}_{\infty}^{\gamma(s,t)}(P(\delta)) \geq \delta^{\xi}, \end{equation}
assuming that $\epsilon = \epsilon(\xi) > 0$ and the upper bound $\delta_{0} = \delta(\xi) > 0$ for the scale $\delta$ were chosen small enough. The claim in the lemma immediately follows from \eqref{form33}, and \cite[Proposition A.1]{FasslerOrponen14}. This proposition, in general, states that if $B \subset \R^{d}$ is a set with $\mathcal{H}^{s}_{\infty}(B) = \kappa > 0$, then $B$ contains a non-empty $(\delta,s,C\kappa^{-1})$-set for some absolute constant $C > 0$. In particular, from \eqref{form33} we see that $P(\delta)$ contains a $(\delta,\gamma(s,t),C\delta^{-\xi})$-set. This easily implies a similar conclusion about $P$ itself. \end{proof}

By standard point-line duality considerations (see a few details below the statement), Lemma \ref{lemma4} is equivalent to the following statement concerning tubes:
\begin{lemma}\label{lemma5} For every $\xi > 0$, there exists $\delta_{0} = \delta_{0}(\xi) > 0$ and $\epsilon = \epsilon(\xi) > 0$ such that the following holds for all $\delta \in (0,\delta_{0}]$. Let $s \in [0,1]$ and $t \in [0,2]$. Assume that $P \subset B^{2}$ is a non-empty $(\delta,t,\delta^{-\epsilon})$-set. Assume that for every $x \in P$ there exists a non-empty $(\delta,s,\delta^{-\epsilon})$-set of tubes $\mathcal{T}_{x}$ with the property that $x \in T$ for all $T \in \mathcal{T}_{x}$. Then, the union
\begin{equation}\label{form32} \mathcal{T} := \bigcup_{x \in P} \mathcal{T}_{x} \end{equation}
contains a non-empty $(\delta,\gamma(s,t),\delta^{-\xi})$-set, where $\gamma(s,t) = s + \min\{s,t\}$. \end{lemma}

If the reader is not familiar with point-line duality, then the full details in a very similar context are recorded in \cite[Sections 6.1-6.2]{2021arXiv210704471D}. Here we just describe the key ideas. To every point $(a,b) \in \R^{2}$, we associate the line $\mathbf{D}(a,b) := \{y = ax + b : x \in \R\} \in \mathcal{A}(2,1)$. Conversely, to every line $\ell = \{y = cx + d : x \in \R\}$ we associate the point $\mathbf{D}^{\ast}(\ell) = (-c,d)$.
Then, it is easy to check that
\begin{equation}\label{form34} p \in \ell \quad \Longleftrightarrow \quad \mathbf{D}^{\ast}(\ell) \in \mathbf{D}(p). \end{equation}
For $(a,b),(c,d) \in [0,1]^{2}$, say, the maps $\mathbf{D}$ and $\mathbf{D}^{\ast}$ are bilipschitz between the Euclidean metric, and the metric on $\mathcal{A}(2,1)$. Therefore the property of "being a $(\delta,s)$-set" is preserved (up to inflating the constants). Now, roughly speaking, Lemma \ref{lemma5} follows from Lemma \ref{lemma4} by first applying the transformations $\mathbf{D},\mathbf{D}^{\ast}$ to the points $P$ and the tubes $\mathcal{T}_{x}$, $x \in P$, respectively. The main technicalities arise from the fact that $\mathcal{T}_{x}$ is a set of $\delta$-tubes, and not a set of lines. Let us ignore this issue for now, and assume that $\mathcal{T}_{x} = \mathcal{L}_{x}$ is actually a $(\delta,s)$-set of lines such that $x \in \ell$ for all $\ell \in \mathcal{L}_{x}$. In this case Lemma \ref{lemma5} is simple to infer from Lemma \ref{lemma4}.

Write $P = \mathbf{D}^{\ast}(\mathcal{L})$ for some $(\delta,t)$-set of lines $\mathcal{L} \subset \mathcal{A}(2,1)$, and write also $\mathcal{L}_{x} = \mathbf{D}(P_{x})$ for some $(\delta,s)$-set of points $P_{x} \subset \R^{2}$. Now, if $\ell \in \mathcal{L}$, then $\mathbf{D}^{\ast}(\ell) = x \in \mathbf{D}(y)$ for all $y \in P_{x}$ by assumption. By \eqref{form34}, this is equivalent to $P_{x} \subset \ell$. Thus, every line $\ell = (\mathbf{D}^{\ast})^{-1}(x) \in \mathcal{L}$, $x \in P$, contains a $(\delta,s)$-set $P_{x} =: P_{\ell}$. This places us in a position to apply Lemma \ref{lemma4}.

A similar argument still works if $\mathcal{L}_{x}$ is replaced by the $(\delta,s)$-set of tubes $\mathcal{T}_{x}$. One only needs to make sure that if $x \in T \in \mathcal{T}_{x}$, then the line $\ell = (\mathbf{D}^{\ast})^{-1}(x)$ is $O(\delta)$-close to a certain $(\delta,s)$-set $P_{\ell}$; this set can be derived from $\mathcal{T}_{x}$ by using the idea above. For the technical details, we refer to \cite[Sections 6.1-6.2]{2021arXiv210704471D}, in particular \cite[Lemma 6.7]{2021arXiv210704471D}.

We are finally equipped to prove Theorem \ref{main2}:

\begin{proof}[Proof of Theorem \ref{main2}]
Fix $s \in [0,2]$, $t \in (1,2]$, $\sigma \in [0,1)$, and $\zeta > 0$. Let $P_{K},P_{E} \subset B^{2}$ be as in the statement of the theorem: thus $P_{K}$ is a $(\delta,t,\delta^{-\epsilon})$-set, and $P_{E}$ is a $(\delta,s,\delta^{-\epsilon})$-set. Recall also the $(\delta,\sigma,\delta^{-\epsilon})$-sets of tubes $\mathcal{T}_{x}$ passing through $x$, for every $x \in P_{E}$, with the property
\begin{equation}\label{form28} |T\cap P_{K}| \geq \delta^{\sigma + \epsilon}|P_{K}|, \quad T \in \mathcal{T}_{x}. \end{equation}
The claim is that
\begin{equation}\label{form30} \sigma \geq s + t - 1 - \zeta \end{equation}
if $\delta,\epsilon > 0$ are chosen small enough, depending only on $\zeta,\sigma,t$.

 By Lemma \ref{lemma5}, the union $\bigcup_{x \in P_{E}} \mathcal{T}_{x}$ contains a non-empty $(\delta,\gamma(\sigma,s),\delta^{-\xi})$-set $\mathcal{T}$, where
 \begin{displaymath} \gamma(\sigma,s) = \sigma + \min\{s,\sigma\}, \end{displaymath}
 and $\xi > 0$ can be made as small as we like by choosing $\epsilon,\delta > 0$ sufficiently small (independently of $s,\sigma$). Now we are prepared to spell out all the requirements on $\epsilon$:
 \begin{equation}\label{form74} 10\xi + 2\epsilon \leq \zeta \quad \text{and} \quad \sigma < 1 - 5\xi - \epsilon \quad \text{and} \quad (t - 1)/2 - 5\xi - \epsilon > 0. \end{equation}
 We may assume that $\xi \geq \epsilon$, so our $(\delta,t,\delta^{-\epsilon})$-set $P_{K}$ is also a $(\delta,t,\delta^{-\xi})$-set (if $\xi < \epsilon$, then both $P_{K}$ and $\mathcal{T}$ are $(\delta,u,\delta^{-\epsilon})$-sets with, and this would work even better in the sequel).

 By \eqref{form28}, we have
 \begin{equation}\label{form29} |P_{K}||\mathcal{T}| \cdot \delta^{\sigma + \epsilon} \leq \sum_{T \in \mathcal{T}} |T\cap P_{K}| = |\mathcal{I}(P_{K},\mathcal{T})|. \end{equation}
 We next compare this lower bound for $|\mathcal{I}(P_{K},\mathcal{T})|$ against the upper bounds from Theorem \ref{FuRen}. Recall the exponent "$\kappa$" from the statement of Theorem \ref{FuRen}. Since $P_{K}$ is a $(\delta,t)$-set and $\mathcal{T}$ is a $(\delta,\gamma(\sigma,s))$-set, the useful quantity for us is
  \begin{displaymath} \bar{\kappa}(s,\sigma,t) := \kappa(t,\gamma(\sigma,s)) = \min\{1/2,1/(t + \gamma(\sigma,s) - 1)\}. \end{displaymath}
The remainder of the proof splits into four cases:
 \begin{enumerate}[(i)]
 \item Assume first that $s \leq \sigma$. Thus $\gamma(\sigma,s) = s + \sigma$, so $\mathcal{T}$ is a $(\delta,s + \sigma,\delta^{-\xi})$-set.
 \begin{itemize}
  \item[(a)] Assume that $\bar{\kappa}(s,\sigma,t) = 1/2$. Then, by Theorem \ref{FuRen},
  \begin{displaymath} |\mathcal{I}(P_{K},\mathcal{T})| \leq |P_{K}||\mathcal{T}| \cdot \delta^{(t + (s + \sigma) - 1)/2 - 5\xi}. \end{displaymath}
  Comparing this against \eqref{form29} yields $\delta^{2\sigma+2\epsilon} \leq \delta^{t + s + \sigma - 1 - 10\xi}$, and therefore $\sigma \geq s + t - 1 - 10\xi - 2\epsilon$. This yields \eqref{form30}, since we assumed that $10\xi + 2\epsilon \leq \zeta$.
  \item[(b)] Assume that $\bar{\kappa}(s,\sigma,t) = 1/(t + s + \sigma - 1)$. Then,
  \begin{displaymath} |\mathcal{I}(P_{K},\mathcal{T})| \leq |P_{K}||\mathcal{T}| \cdot \delta^{1 - 5\xi}. \end{displaymath}
  Comparing against \eqref{form29} yields $\delta^{\sigma} \leq \delta^{1 - 5\xi - \epsilon}$, contradicting \eqref{form74}.
 \end{itemize}
 \item Assume second that $s > \sigma$. Thus $\gamma(\sigma,s) = 2\sigma$, so $\mathcal{T}$ is a $(\delta,2\sigma,\delta^{-\xi})$-set.
 \begin{itemize}
 \item[(a)] Assume that $\bar{\kappa}(s,\sigma,t) = \tfrac{1}{2}$. Then,
 \begin{displaymath} |\mathcal{I}(P_{K},\mathcal{T})| \leq |P_{K}||\mathcal{T}| \cdot \delta^{(t + 2\sigma - 1)/2 - 5\xi}. \end{displaymath}
 Comparing this against \eqref{form29} yields $1 \leq \delta^{(t - 1)/2 - 5\xi - \epsilon}$, contradicting \eqref{form74}.
 \item[(b)] Assume finally that $\bar{\kappa}(s,\sigma,t) = 1/(t + 2\sigma - 1)$. Then,
 \begin{displaymath} |\mathcal{I}(\mathcal{P}_{K},\mathcal{T})| \leq |P_{K}||\mathcal{T}| \cdot \delta^{1 - 5\xi}. \end{displaymath}
 As in case (i)(b) above, this leads to the impossible situation $\delta^{\sigma} \leq \delta^{1 - 5\xi - \epsilon}$.
 \end{itemize}
 \end{enumerate}
 We have now seen that the cases (i)(b) and (ii)(a)-(b) are not possible for $\delta,\epsilon$ small enough, depending only on $\zeta > 0$, $\sigma < 1$ and $t > 1$. Case (i)(a), on the other hand, yields the desired inequality \eqref{form30} for $10\xi + 2\epsilon \leq \zeta$. This completes the proof of Theorem \ref{main2}. \end{proof}

 \subsection{Proof of Theorem \ref{main}}

Theorem \ref{main} follows immediately from the following "thin tubes version", whose proof is further based on Theorem \ref{main2} from the previous section.

 \begin{thm}\label{mainThinTubes} Let $s \in [0,2]$, $t \in (1,2]$, $0 \leq \sigma < \min\{s + t - 1,1\}$, $C > 0$ and $\epsilon \in (0,1]$. Then, there exists $K = K(C,\epsilon,s,\sigma,t) > 0$ such that the following holds. Assume that $\mu,\nu \in \mathcal{P}(B^{2})$ satisfy $\mu(B(x,r)) \leq Cr^{s}$ and $\nu(B(x,r)) \leq Cr^{t}$, or alternatively $I_{s}(\mu) \leq C$ and $I_{t}(\nu) \leq C$. Then $(\mu,\nu)$ has $(\sigma,K,1 - \epsilon)$-thin tubes.

In particular, whenever $s \in [0,2]$, $t \in (1,2]$, $I_{s}(\mu) < \infty$ and $I_{t}(\nu) < \infty$, then $(\mu,\nu)$ has $\sigma$-thin tubes for every $0 \leq \sigma < \min\{s + t - 1,1\}$.
\end{thm}

The proof of Theorem \ref{mainThinTubes} is similar to the proof of Corollary \ref{cor:radial}. We establish an "$\epsilon$-improvement" version of the result, Lemma \ref{lemma7} below, which can then be iterated multiple times to derive Theorem \ref{mainThinTubes}.

 \begin{lemma}\label{lemma7}
	Let $s \in [0,2]$, $t \in (1,2]$, and $0 \leq \sigma < \min\{s + t - 1,1\}$. Let $\epsilon \in (0,\tfrac{1}{10})$ and $C,K > 0$. Let $\mu,\nu\in\mathcal{P}(B^2)$ such that $\mu(B(x,r)) \leq Cr^{s}$ and $\nu(B(y,r)) \leq Cr^{t}$ for all $x,y \in \R^{2}$ and $r > 0$. If $(\mu, \nu)$  has $(\sigma, K, 1 - \epsilon)$-thin tubes, then there exist $\eta = \eta(s,\sigma,t) > 0$ and  $K' = K'(C,K,\epsilon,\sigma, s,t) > 0$ such that $(\mu, \nu)$ has $(\sigma+\eta, K', 1 - 4\epsilon)$-thin tubes. Moreover, $\eta(s,\sigma,t)$ is bounded away from zero on any compact subset of
\begin{equation}\label{form75} \Omega := \{(s,\sigma,t) \in [0,2] \times [0,1) \times (1,2] : \sigma < \min\{s + t - 1,1\}\}. \end{equation}
\end{lemma}

Before proving Lemma \ref{lemma7}, we complete the proof of Theorem \ref{mainThinTubes}:

\begin{proof}[Proof of Theorem \ref{mainThinTubes} assuming Lemma \ref{lemma7}] The starting point is that $(\mu,\nu)$ has $(t - 1,K_{0},1)$-thin tubes for some $K_{0} \sim C$ by the Frostman condition on $\nu$ alone: $\nu(T) \lesssim C \cdot r^{t - 1}$ for all $r$-tubes $T \subset \R^{2}$. In particular, $(\mu,\nu)$ has $(t - 1,K_{0},1 - \bar{\epsilon})$-thin tubes for every $\bar{\epsilon} \in (0,\tfrac{1}{10})$. If $s = 0$ or $t = 2$, we are done. Otherwise $\min\{s + t - 1,1\} > t - 1$, and we need to apply Lemma \ref{lemma7} a few times. Fix $0 \leq \sigma < \min\{s + t - 1,1\}$, and let
\begin{displaymath} \eta := \eta(s,\sigma,t) := \inf \{\eta(s,\sigma',t) : t - 1 \leq \sigma' \leq \sigma\}, \end{displaymath}
where $\eta(s,\sigma',t)$ is the function in Lemma \ref{lemma7}. We have $\eta > 0$, since $\eta(s,\sigma',t)$ is bounded away from zero on compact subsets of $\Omega$, as in \eqref{form75}. We also choose
\begin{displaymath} \bar{\epsilon} := \epsilon \cdot 4^{-1/\eta}/100. \end{displaymath}
where $\epsilon > 0$ is the constant given in the statement.

The first application of Lemma \ref{lemma7} implies that $(\mu,\nu)$ has $(t - 1 + \eta,K_{1},1 - 4\bar{\epsilon})$-thin tubes for some $K_{1} = K_{1}(C,\epsilon,s,\sigma,t) > 0$.\footnote{The upper bound for "$K_{1}$" in Lemma \ref{lemma7} also depends on the lower bound for $\epsilon > 0$, so the argument here does not show that $(\mu,\nu)$ has $(t - 1 + \eta,K_{1},1)$-thin tubes. This would indeed be false in general.} If $t - 1 + \eta > \sigma$, we are done. Otherwise, a second application of Lemma \ref{lemma7} shows that $(\mu,\nu)$ has $(t - 1 + 2\eta,K_{2},1 - 4^{2}\bar{\epsilon})$-thin tubes for some $K_{2} > 0$. We proceed in the same manner. After $N \leq 1/\eta$ steps, we find that $(\mu,\nu)$ has $(\sigma,K_{N},1 - 4^{N}\bar{\epsilon})$-thin tubes for some $K_{N} > 0$, and the iteration terminates. At this point, notice that $4^{N}\bar{\epsilon} \leq 4^{1/\eta}\bar{\epsilon} \leq \epsilon$. Also, $K_{N}$ only depends on $C,\epsilon,s,\sigma,t$ and $N = N(\eta) = N(s,\sigma,t)$, as desired. This completes the proof. \end{proof}

Finally, we prove Lemma \ref{lemma7}:

\begin{proof}[Proof of Lemma \ref{lemma7}]
The scheme of the proof is very similar to that of Lemma \ref{lem: radial2}. The main difference lies in the geometric input, which is provided by Theorem \ref{main2}.

We argue by contradiction. Assume  that $(\mu, \nu)$ does not have $(\sigma+\eta, K', 1-4\epsilon)$-thin tubes for $K'=K'(C,K,\epsilon,\sigma,s,t) \ge 1$ and $\eta = \eta(s,\sigma,t) > 0$ to be determined in the course of the proof. In fact, we explain immediately the dependence of $\eta$ on $s,t,\sigma$. Since $\sigma < s + t - 1$, we may choose $\zeta = \zeta(s,\sigma,t) > 0$ such that
\begin{equation}\label{form69} \sigma < s + t - 1 - \zeta. \end{equation}
The proof will be concluded by applying Theorem \ref{main2} with the parameters $s,t$, and this $\zeta(s,\sigma,t) > 0$. Theorem \ref{main2} gives us a parameter $\eta' = \eta'(\sigma,t,\zeta) = \eta'(s,\sigma,t) > 0$ associated with these constants. Our choice of $\eta = \eta(s,t,\sigma)$ needs to be so small that $C\eta < \eta'$ for a certain absolute constant $C > 0$. Since the value of $\eta'(\sigma,t,\zeta)$ is bounded away from zero on compact subsets of $[0,1) \times (1,2] \times [0,1)$, the value of $\eta$ will be (or can be taken to be) bounded away from zero on compact subsets of the set $\Omega$ in \eqref{form75}.

We then begin the proof in earnest. Write $X=\spt(\mu)$, $Y=\spt(\nu)$. Since $(\mu, \nu)$ has $(\sigma, K, 1-\epsilon)$--thin tubes, there is a Borel set $G \subset X \times Y$ with $(\mu\times\nu)(G)>1- \epsilon$ such that
\begin{equation}\label{form66a}
\nu(T\cap G|_x) \le K\cdot r^{\sigma} \text{ for all } r > 0 \text{ and  all } r\text{-tubes } T \text{ containing } x\in X. \end{equation}
For every $x \in X$ and $r \in 2^{-\N}$, let $\mathcal{T}_{x,r}''$ consist of those $r$-tubes, with $r \in 2^{-\N}$, which contain $x$ and satisfy
\begin{equation}\label{form67} \nu(T\cap G|_{x}) \geq \tfrac{K'}{2} \cdot r^{\sigma + \eta}. \end{equation}
We pick a maximal subset of $\mathcal{T}_{x,r}''$ with $r$-separated angles, and inflate each element of this subset by a factor of $10$. Let $\mathcal{T}_{x,r}'$ be the family of $10r$-tubes so obtained. Evidently,
\begin{equation}\label{form62a} (\cup \mathcal{T}_{x,r}'') \cap B^{2} \subset \cup \mathcal{T}_{x,r}' \quad \text{and} \quad |\mathcal{T}_{x,r}'| \lesssim r^{-\sigma - \eta}. \end{equation}
We define $\bar{H}_{r} := \{(x,y) \in X \times Y : y \in \cup \mathcal{T}_{x,r}'\}$ and $\bar{H} := \bigcup_{r} \bar{H}_{r}$. We claim that
\begin{displaymath} (\mu \times \nu)(\bar{H}) \geq 2\epsilon. \end{displaymath}
Indeed, assume to the contrary that $(\mu \times \nu)(\bar{H}) < 2\epsilon$, thus $(\mu \times \nu)(\bar{H}^{c}) > 1 - 2\epsilon$, and
\begin{displaymath} (\mu \times \nu)(\bar{H}^{c} \cap G) \geq 1 - 3\epsilon. \end{displaymath}
Since $(\mu,\nu)$ does not have $(\sigma + \eta,K',1 - 4\epsilon)$-thin tubes by assumption, we infer that there exists a point $x \in X$, and an $r$-tube $T \subset \R^{2}$, $r \in 2^{-\N}$, containing $x$ which satisfies
\begin{displaymath} \nu(T\cap G|_{x}) \geq \nu(T\cap (\bar{H}^{c} \cap G)|_{x}) \geq \tfrac{K'}{2} \cdot r^{\sigma + \eta}. \end{displaymath}
However, this means by the definition of $\mathcal{T}_{x,r}''$, and the inclusion in \eqref{form62a}, that $T \cap B^{2} \subset \bar{H}|_{x}$, so it is absurd that $\nu(T\cap (\bar{H}^{c})|_{x}) > 0$. This contradiction completes the proof of $(\mu \times \nu)(\bar{H}) \geq 2\epsilon$.

We now set
\begin{displaymath} H := G \cap \bar{H} \quad \text{and} \quad H_{r} := G \cap \bar{H}_{r}, \end{displaymath}
and we infer from a combination of $(\mu \times \nu)(G) \geq 1 - \epsilon$ and $(\mu \times \nu)(\bar{H}) \geq 2\epsilon$ that
\begin{displaymath} (\mu \times \nu)(H) \geq \epsilon. \end{displaymath}
Fix $r_0=r_0(C,K,\epsilon,s,\sigma)\in 2^{-\N}$ such that $\sum_{r\le r_0} r^\eta < \epsilon/2$ (recall that $\eta$ was determined by $s,\sigma,t$). Later we will impose further upper bounds on $r_0$, always depending on $C,K,\epsilon,s,\sigma,t$ only. If $K'$ is taken so large that $(K'/2) \cdot r_0^{\sigma+\eta}>1$ (hence in terms of $C,K,\epsilon,s,\sigma,t$ only), then  we see from \eqref{form67} that $\mathcal{T}_{x,r}' =\emptyset$ for all $r > r_0$. This in particular implies that $H_{r}|_{x} \subset \cup \mathcal{T}_{x,r}'$ is empty for all $x \in X$ and $r > r_{0}$. In other words $H_{r} = \emptyset$ for $r > r_{0}$, which implies that $H$ is contained in the union of the sets $H_{r}$ with $r \leq r_{0}$. Since $(\mu \times \nu)(H) \geq \epsilon$, it now follows from the choice of $r_{0}$ that
\begin{displaymath} (\mu \times \nu)(H_{r}) \geq 2r^{\eta} \quad \text{for some} \quad r \leq r_{0}. \end{displaymath}
We will fix this value of "$r$" for the remainder of the proof. We define
\begin{displaymath} \mathbf{X} := \{x \in X : \nu(H_{r}|_{x}) \geq r^{\eta}\}, \end{displaymath}
and note that $\mu(\mathbf{X}) \geq r^{\eta}$. Recall from the definition of $\bar{H}_{r} \supset H_{r}$ that the fibres $H_{r}|_x$ are covered by the tube families $\mathcal{T}_{x,r}'$, and by \eqref{form62a} that $|\mathcal{T}_{x,r}'| \lesssim r^{ -\sigma -\eta}$.  For $x\in \mathbf{X}$, define
\[
\mathcal{T}_x := \{ T\in\mathcal{T}'_{x,r}: \nu(T\cap H_{r}|_{x}) \geq r^{\sigma+3\eta}\},\quad Y_x = (H_{r}|_{x}) \cap \left( \cup \mathcal{T}_x \right).
\]
Since $\nu(H_{r}|_{x}) \geq r^{\eta}$ for $x \in \mathbf{X}$, we have $\nu(Y_{x}) \geq r^{2\eta}$ for all $x \in \mathbf{X}$. For every $x \in \mathbf{X}$, the set $\mathcal{T}_x$ of $r$-tubes covers $Y_x$ and satisfies
\begin{equation}\label{form68}
r^{\sigma + 3\eta} \leq  \nu(T\cap Y_x)   \leq r^{\sigma-\eta}, \qquad T\in \mathcal{T}_x,
\end{equation}
where the upper bound follows from $Y_{x} \subset H|_{x} \subset G|_{x}$ and \eqref{form66a}. In fact, more generally
\[
\nu(T^{(\rho)}\cap Y_x) \leq r^{-\eta}\rho^{\sigma}, \qquad \rho\in [r, 1], \, T\in \mathcal{T}_x.
\]
Putting these facts together, we see that
\begin{equation} r^{-\sigma + 3\eta} \leq r^{-\sigma + \eta} \cdot \nu(Y_{x})  \lesssim |\mathcal{T}_{x}| \leq |\mathcal{T}_{x,r}'| \lesssim r^{-\sigma - \eta}, \qquad x \in \mathbf{X}, \end{equation}
and $\mathcal{T}_{x}$ is an $(r, \sigma, r^{-4\eta})$--set for $x \in \mathbf{X}$. 

We now discretise everything at scale $r$. First, apply the discrete Frostman's Lemma \cite[Proposition A.1]{FasslerOrponen14} to obtain an $(r,s,r^{-O(\eta)})$-set $P_{X} \subset \mathbf{X}$. We would also like to select an $(r,t,r^{-O(\eta)})$-set $P_{Y} \subset Y$, but we need to do this a little more carefully. For each index $j \geq 0$, let
\begin{displaymath} Y_{j} := \{y \in Y : 2^{-j - 1} \cdot Cr^{t} < \nu(B(y,r)) \leq 2^{-j} \cdot Cr^{t}\}. \end{displaymath}
Then $Y$ is contained in the union of the sets $Y_{j}$. In particular, each of the sets $Y_{x}$,  $x \in P_{X}$, is contained in this union. Therefore, for each $x \in P_{X}$, there exists $j = j(x) \geq 0$ such that $\nu(Y_{x} \cap Y_{j}) \geq r^{3\eta}$ (since there are only $\leq C\log(1/r) \leq r^{-\eta}$ choices of "$j$" which need to be considered here). By the pigeonhole principle, we may then find a subset $P_{X}' \subset P_{X}$ such that $|P_{X}'| \geq r^{\eta}|P_{X}|$, and $j(x) = j$ is constant for $x \in P_{X}'$. Since $P_{X}'$ remains an $(r,s,r^{-O(\eta)})$-set, we keep denoting $P_{X}'$ by $P_{X}$ in the sequel.

For the index "$j$" found above, let $P_{Y} \subset Y_{j}$ be a maximal $r$-separated set. Obviously $|P_{Y}| \lesssim 2^{j} \cdot C r^{-t}$, but since, for any $x\in P_X$,
\begin{displaymath} r^{3\eta} \leq  \nu(Y_{x} \cap Y_{j}) \lesssim (2^{-j} \cdot Cr^{t}) \cdot |P_{Y}|, \end{displaymath}
we also have the nearly matching lower bound $|P_{Y}| \gtrsim C^{-1} \cdot 2^{j}r^{-t + 3\eta}$. After this observation, it follows from the calculation
\begin{displaymath} |P_{Y} \cap B(y,\rho)| \cdot (2^{-j} \cdot Cr^{t}) \lesssim \nu(B(y,2\rho)) \leq C \cdot (2\rho)^{t}, \qquad y \in \R^{2}, \, \rho \geq r, \end{displaymath}
that $P_{Y}$ is an $(r,t,r^{-O(\eta)})$-set (we suppress the dependence on "$C$" by choosing $r$ smaller in a manner depending on $C$).

We next study the cardinality of $P_{Y}$ inside the tubes $2T$ for $T \in \mathcal{T}_{x}$, $x \in P_{X}$. Fix $x \in P_{X}$. Recall that
\begin{displaymath} Y_{x} \cap Y_{j} \subset Y_{x} \subset \cup \mathcal{T}_{x}, \end{displaymath}
and $|\mathcal{T}_{x}| \lesssim r^{-\sigma - \eta}$. Since $\nu(Y_{x} \cap Y_{j}) \geq r^{3\eta}$, there is a subset of $\mathcal{T}_{x}' \subset \mathcal{T}_{x}$ of cardinality $|\mathcal{T}_{x}'| \geq r^{O(\eta)} \cdot |\mathcal{T}_{x}|$ such that $\nu(Y_{x} \cap Y_{j} \cap T) \geq r^{\sigma + O(\eta)}$ for all $T \in \mathcal{T}_{x}'$. Evidently $\mathcal{T}_{x}'$ remains a $(r,\sigma,r^{-O(\eta)})$-set. Now, if $T \in \mathcal{T}_{x}'$, we have
\begin{displaymath} r^{\sigma + O(\eta)} \leq \nu(Y_{x} \cap Y_{j} \cap T) \lesssim (2^{-j} \cdot Cr^{t}) \cdot |P_{Y} \cap 2T|, \end{displaymath}
and therefore $|P_{Y} \cap 2T| \gtrsim r^{\sigma + O(\eta)} \cdot (2^{j}r^{-t}) \gtrsim r^{\sigma + O(\eta)}|P_{Y}|$. Finally, define
\begin{displaymath} \mathcal{T}_{x}'' := \{2T : T \in \mathcal{T}_{x}'\}. \end{displaymath}
Then $\mathcal{T}_{x}''$ is an $(r,\sigma,r^{-O(\eta)})$-set of $2r$-tubes for all $x \in P_{X}$, where $P_{X}$ is an $(r,s,r^{-O(\eta)})$-set. Moreover, $|T\cap P_{Y} | \geq r^{\sigma + O(\eta)}|P_{Y}|$ for all $T \in \mathcal{T}_{x}''$, where $P_{Y}$ is an $(r,t,r^{-O(\eta)})$-set. These are precisely the hypotheses of Theorem \ref{main2}. Thus, if $\eta,r > 0$ are small enough, depending on $\sigma,t,\zeta$ (therefore $s,\sigma,t$), it follows that $\sigma \geq s + t - 1 - \zeta$. This contradicts our choice of "$\zeta$" at \eqref{form69}, and completes the proof.  \end{proof}

\section{Higher dimensions}\label{s:higherDim}

In this section, we give new proofs of Theorems \ref{ganDote1}-\ref{ganDote2}, due to Dote and Gan, by reducing them to planar statements, more precisely Corollary \ref{cor:radial2} and Theorem \ref{mainThinTubes}. In fact, as discussed in the introduction, we prove a sharper version of Theorem \ref{ganDote2} that can be seen as a partial extension of Theorem \ref{mainNew} to higher dimensions. We first recall the statements:

\begin{thm}\label{gd1} Let $Y \subset \R^{d}$ be a Borel set with $\Hd Y \in (k,k + 1]$ for some $k \in \{1,\ldots,d - 1\}$. Then,
\begin{displaymath} \Hd \{x \in \R^{d} : \Hd \pi_{x}(Y \, \setminus \, \{x\}) < \sigma\} \leq \max\{k + \sigma - \Hd Y,0\}, \qquad 0 \leq \sigma < k. \end{displaymath}
 \end{thm}

\begin{thm}\label{gd2}
Let $X,Y\subset\R^d$ be Borel sets with $\Hd X > k - 1$ and $\Hd Y \in (k - 1,k]$ for some $k\in\{1,\ldots,d-1\}$.
\begin{itemize}
\item[\textup{(i)}] If $\Hd X > k$, then $\sup_{x \in X} \Hd \pi_{x}(Y \, \setminus \, \{x\}) = \Hd Y$.
\item[\textup{(ii)}] If $k - 1 < \Hd X \leq k$, but $X$ is not contained on any $k$-plane, the following holds. If $\Hd Y > k - 1/k - \eta$ for a sufficiently small constant $\eta = \eta(d,k,\Hd X) > 0$, then
\begin{displaymath} \sup_{x \in X} \Hd \pi_{x}(Y \, \setminus \, \{x\}) \geq \min\{\Hd X,\Hd Y\}. \end{displaymath}
\end{itemize}
For $k = 1$, we require no lower bound from $\Hd Y$ in part \textup{(ii)}.
\end{thm}

We introduce some additional notation. The Grassmanian of linear $m$-planes in $\R^d$ will be denoted by $\mathcal{G}(d,m)$. We endow $\mathcal{G}(d,m)$ with the natural Borel probability measure $\gamma_{d,m}$ invariant under the action of the orthogonal group - see \cite[Chapter 3]{zbMATH01249699}. The orthogonal projection onto $V\in\mathcal{G}(d,m)$ is denoted by $P_{V}$. Similarly, the Grassmanian of \emph{affine} $m$-planes in $\R^d$ is denoted by $\mathcal{A}(d,m)$ and the natural isometry-invariant measure on it by $\lambda_{d,m}$ - see \cite[\S 3.16]{zbMATH01249699} for its definition. Given a finite Borel measure $\mu$ on $\R^d$, the sliced measures of $\mu$ on the affine planes $V+x$ (where $V\in\mathcal{G}(d,m)$ and $x\in V^\perp$) are denoted by $\mu_{V,x}$. We extend the definition to $x\in\R^d$ by setting $\mu_{V,x}:=\mu_{V,P_{V^\perp}x}$. See \cite[Chapter 10]{zbMATH01249699} for the definition and key properties of sliced measures. Finally, to lighten up notation we denote $\pi_y(X)=\pi_y(X \, \setminus \, \{y\})$.

We start by reducing Theorems \ref{gd1}--\ref{gd2} to the special case $k = d - 1$. Similar arguments, involving lifting radial projection estimates from  a random projection to a suitable lower dimensional subspace, appeared earlier in \cite{DIOWZ21, 2021arXiv211209044S}.
\begin{proposition} Theorems \ref{gd1}-\ref{gd2} follow from their special case $k = d - 1$.
\end{proposition}

\begin{proof} Let $\Sigma(x) := x/|x|$ for $x \in \R^{d} \, \setminus \, \{0\}$. We start by proving the following inclusion, valid for every $m$-plane $V \in \mathcal{G}(d,m)$, and for every $0 < m < d - 1$:
\begin{equation}\label{form64} \pi_{P_{V}(x)}(P_{V}(Y)) \subset \Sigma(P_{V}(\pi_{x}(Y))), \qquad x \in \R^{d}. \end{equation}
To see this, fix $x \in \R^{d}$ and $e \in \pi_{P_{V}(x)}(P_{V}(Y)) \in S^{d - 1}$. Thus, there exists $P_{V}(y) \in P_{V}(Y)$, with $y \in Y$, such that $P_{V}(y) \neq P_{V}(x)$ (in particular $y \neq x$), and
\begin{displaymath} e = \frac{P_{V}(y) - P_{V}(x)}{|P_{V}(y) - P_{V}(x)|} = P_{V}\left(\frac{y - x}{|P_{V}(y - x)|} \right) = \frac{1}{|P_{V}((y - x)/|y - x|)|} \cdot P_{V}\left(\frac{y - x}{|y - x|} \right). \end{displaymath}
The right hand side is an element of $\Sigma(P_{V}(\pi_{x}(Y)))$, as claimed.

We then prove the proposition. More precisely, we will establish the following: let $0 < k < d$, and assume that Theorems \ref{gd1}-\ref{gd2} hold in $\R^{d}$ for this "$k$". Then they also hold in $\R^{d + 1}$ with the same value of "$k$". In particular, if the theorems hold for the value $k = d - 1$ in some $\R^{d}$, then they also hold for $k = d - 1$ in every $\R^{D}$ for $D \geq d$.

We start by establishing the statement above for Theorem \ref{gd2}. Parts (i) and (ii) are very similar, but (ii) is slightly harder, so we spell out the details for (ii). Assume to the contrary that there exist Borel sets $X,Y \subset \R^{d + 1}$ with
\begin{displaymath} \Hd X \in (k - 1,k] \quad \text{and} \quad \Hd Y \in (k - 1/k - \eta,k] \end{displaymath}
such that $X$ is not contained on any $k$-plane, and
\begin{displaymath} \sup_{x \in X} \Hd \pi_{x}(Y) < \min\{\Hd X,\Hd Y\}. \end{displaymath}
In order to apply the "known" $\R^{d}$-version of Theorem \ref{gd2} (and then reach a contradiction), we plan to find a suitable orthogonal projection to a $d$-plane $V \in \mathcal{G}(d + 1,d)$. The Marstrand-Mattila projection theorem, \cite[Corollary 9.4]{zbMATH01249699}, shows that for $\gamma_{d+1,d}$-almost every $V \in \mathcal{G}(d + 1,d)$ we have
\begin{equation}\label{form78} \Hd P_{V}(X) = \Hd X \quad \text{and} \quad \Hd P_{V}(Y) = \Hd Y. \end{equation}
On the other hand, \eqref{form64} shows that
\begin{align*} \sup_{x \in X} \Hd \pi_{P_{V}(x)}(P_{V}(Y)) & \leq \sup_{x \in X} \Hd \pi_{x}(Y)\\
& < \min\{\Hd X,\Hd Y\}\\
& \stackrel{\eqref{form78}}{=} \min\{\Hd P_{V}(X),\Hd P_{V}(Y)\}. \end{align*}
This violates the $\R^{d}$-version of Theorem \ref{gd2} in $V \cong \R^{d}$, applied with $P_{V}(X)$ and $P_{V}(Y)$, except for one problem: $P_{V}(X)$ may be contained in a $k$-plane, even if $X$ is not.  However, we claim that the set of $V \in \mathcal{G}(d + 1,d)$ such that this happens has zero $\gamma_{d + 1,d}$ measure. Indeed, pick $x_0,x_1,\ldots,x_{k+1}\in X$ such that the plane $W$ spanned by $\{x_j-x_0\}_{j=1}^{k+1}$ is $(k+1)$-dimensional. Since $\dim W\le d<d+1$, we have $\theta\notin W$ for $\sigma^d$-almost all $\theta\in S^d\subset\R^{d+1}$ (here $\sigma^d$ is the normalized spherical measure). But $\gamma_{d+1,d}$ is the push-forward of $\sigma^{d}$ under $\theta\to \theta^{\perp}$, so $P_V|_W$ is invertible for $\gamma_{d+1,d}$-almost all $V$, and in particular the span of $P_V(X-x_0)$ contains $P_V(W)\in\mathcal{G}(d+1,k+1)$, giving the claim. The proof can now be concluded by choosing $V \in \mathcal{G}(d + 1,d)$ which satisfies \eqref{form78}, and such that $P_{V}(X)$ is not contained in a $k$-plane.



We turn to the proof of Theorem \ref{gd1}. Assume to the contrary that there exist $k \in \{1,\ldots,d - 1\}$, $0 \leq \sigma < k$, and Borel sets $X,Y \subset \R^{d + 1}$ with $\Hd Y \in (k,k + 1]$,
\begin{displaymath} \Hd X > \max\{k + \sigma - \Hd Y,0\} \end{displaymath}
such that $\pi_{x}(Y) \leq \sigma$ for all $x \in X$. Note that
\begin{displaymath} \max\{k + \sigma - \Hd Y,0\} < k < d, \end{displaymath}
since $k\ge 1$ and $\sigma<k<\Hd Y$. In particular, by the Marstrand-Mattila projection theorem, there exists a plane $V \in \mathcal{G}(d + 1,d)$ such that $\Hd P_{V}(Y) = \Hd Y$, and
\begin{equation}\label{form65} \Hd P_{V}(X) > \max\{k + \sigma - \Hd P_{V}(Y),0\}. \end{equation}
On the other hand, \eqref{form64} shows that
\begin{displaymath} \Hd \pi_{P_{V(x)}}(P_{V}(Y)) \leq \Hd \pi_{x}(Y) \leq \sigma, \qquad x \in X. \end{displaymath}
In other words $P_{V}(X) \subset \{z \in V : \Hd \pi_{z}(P_{V}(Y)) \leq \sigma\}$, and therefore
\begin{displaymath} \Hd \{z \in V : \Hd \pi_{z}(P_{V}(Y)) \leq \sigma\} \stackrel{\eqref{form65}}{>} \max\{k + \sigma - \Hd P_{V}(Y),0\}. \end{displaymath}
This contradicts Theorem \ref{gd1} in $V \cong \R^{d}$, and the proof of the proposition is complete. \end{proof}

It remains to prove the $1$-codimensional cases. The constant "$\eta > 0$" in Theorem \ref{gd2}(ii) will arise from an application of \cite[Proposition 6.8]{2021arXiv211209044S}, which we recall here:
\begin{proposition}\label{ShWaProp2} Let $t \in (d - 2,d - 1]$ and $s \in ((d - 1) - 1/(d - 1) - \eta,d - 1]$, where $\eta = \eta(d,t) > 0$ is a sufficiently small constant. Let $\mu,\nu \in \mathcal{P}(\R^{d})$ be measures with disjoint supports such that $I_{t}(\mu) < \infty$ and $I_{s}(\nu) < \infty$. Assume moreover that $\mu(W) = 0 = \nu(W)$ for all $(d - 1)$-planes $V \subset \R^{d}$. Assume also that $\Hd \spt(\nu) < s + \eta$. Then, possibly after restricting $\mu$ and $\nu$ to subsets of positive measure, we have
\begin{displaymath} (\mu \times \gamma_{d,2})\{(x,V) : \mu_{V,x}(\ell)\nu_{V,x}(\ell) = \mu_{V,x}(\R^d)\nu_{V,x}(\R^d) > 0 \text{ for some line } \ell \subset V + x\} = 0. \end{displaymath}
\end{proposition}

We also recall \cite[Theorem 6.7]{2021arXiv211209044S}, which provides the mechanism to upgrade "thin tubes" information from lower to higher dimensions:
\begin{thm} \label{thm:thin-tubes}
Let $\mu,\nu\in\mathcal{P}(\R^d)$ be measures with $I_s(\mu)<\infty$, $I_s(\nu)<\infty$ with $s\in (k,k+1]$ for some $k\in \{1,\ldots,d-2\}$. Suppose that there is $t>0$ such that the sliced measures $(\mu_{W,z},\nu_{W,z})$ have $t$-thin tubes for $(\gamma_{d,d-k}\times\mu)$-almost all $(W,z)$.

Then  $(\mu,\nu)$ have $(k+t)$-thin tubes.
\end{thm}

We begin with Theorem \ref{gd2}, since the numerology is slightly simpler.

\begin{proof}[Proof of Theorem \ref{gd2} for $k = d - 1$] We start by proving the following claim about thin tubes. Let $\mu,\nu$ be  measures on $\R^{d}$ with $I_{t}(\mu)<\infty$ and $I_{s}(\nu)<\infty$, where either
\begin{equation}\label{form79} t \in (d - 2,d - 1] \quad \text{and} \quad s \in \left((d - 1) - \tfrac{1}{d - 1} - \eta,d - 1\right] \end{equation}
with the constant $\eta = \eta(d,t) > 0$ from Proposition \ref{ShWaProp2}, or
\begin{equation}\label{form80} t > d - 1 \quad \text{and} \quad d - 2 < s \leq d - 1. \end{equation}
In the case \eqref{form79}, assume moreover that $\mu(W) = 0 = \nu(W)$ for all $(d - 1)$-planes $W \subset \R^{d}$, and $\Hd \spt(\nu) < s + \eta$. Then $(\mu,\nu)$ has $\sigma$-thin tubes for all $0 \leq \sigma < \min\{s,t\}$.

The case $d=2$ follows immediately from Corollary \ref{cor:radial2}. Consider now the case $d\ge 3$.  Since $t > d - 2$, by the Marstrand-Mattila projection theorem (in this application the absolutely continuous case, \cite[Theorem 9.7]{MR3617376}), the push-forward of $\mu\times\gamma_{d,2}$ under $(x,V)\to V+x$ is absolutely continuous with respect to $\lambda_{d,2}$ (for details, see \cite[Lemma 6.3]{2021arXiv211209044S}). Since also $s > d - 2$, we then deduce from the Marstrand-Mattila slicing theorem, \cite[Theorem 10.7]{zbMATH01249699}, that
\[
I_{s - (d - 2)}(\nu_{V,x})<\infty \quad \text{and} \quad I_{t - (d - 2)}(\mu_{V,x})<\infty \quad\text{for } (\mu\times\gamma_{d,2})\text{-almost all }(x,V).
\]
In the case \eqref{form79}, and under the assumptions $\mu(W) = 0 = \nu(W)$ and $\Hd \spt(\nu) < s + \eta$, we may infer the following from Proposition \ref{ShWaProp2}: for $(\mu \times \gamma_{d,2})$ almost every pair $(x,V)$ such that $\mu_{V,x} \neq 0 \neq \nu_{V,x}$, we have
\begin{displaymath} \mu_{V,x}(\ell)\nu_{V,x}(\ell) < \mu_{V,x}(\R^d)\nu_{V,x}(\R^d) \end{displaymath}
for every line $\ell \subset V + x$. The same is also true in the case \eqref{form80}, for the simpler reason that $t - (d - 2) > 1$: this implies that $\mu_{V,x}(\ell) = 0$ for all lines $\ell \subset V + x$ whenever $I_{t - (d - 2)}(\mu_{V,x}) < \infty$, and thus for $(\mu \times \gamma_{d,2})$ almost every $(x,V)$.

In both cases \eqref{form79}-\eqref{form80}, we have now checked that $\mu_{V,x}$ and $\nu_{V,x}$ satisfy the hypotheses of Corollary \ref{cor:radial2} for $(\mu \times \gamma_{d,2})$ almost every $(x,V)$. Consequently, $(\mu_{V,x},\nu_{V,x})$ has $\sigma$-thin tubes for all $0 \leq \sigma < \min\{s - (d - 2),t - (d - 2)\}$. Finally, it follows from Theorem \ref{thm:thin-tubes}, applied with $k=d-2$, that $(\mu,\nu)$ has $\sigma$-thin tubes for all $0 \leq \sigma < \min\{s,t\}$, as claimed.

We are then equipped to prove the case $k = d - 1$ of Theorem \ref{gd2}. We only spell out the details for Theorem \ref{gd2}(ii), which uses the case \eqref{form79} of the thin tubes statement above. Theorem \ref{gd2}(i) uses the case \eqref{form80} and is substantially simpler.

 Let $X,Y \subset \R^{d}$ be  Borel sets with
 \begin{displaymath} \Hd X \in (d - 2,d - 1] \quad \text{and} \quad \Hd Y \in \left((d - 1) - \tfrac{1}{d - 1} - \eta,d - 1\right]. \end{displaymath}
 Assume that $X$ is not contained in any $(d-1)$-plane, and write $u := \min\{\Hd X,\Hd Y\}$. We claim that $\sup_{x \in X} \Hd \pi_{x}(Y) \geq u$. We first dispose of a special case where
\begin{displaymath} \sup_{W\in\mathcal{A}(d,d-1)} \Hd (Y \cap W) \geq u. \end{displaymath}
Since $\pi_x|_W$ is locally bi-Lipschitz for any $(d-1)$-plane $W$ and $x\in \R^d\setminus W$, and since for any such plane $W$ we may by assumption find $x\in X\setminus W$, we see that $\sup_{x \in X} \Hd \pi_{x}(Y) \geq u$ in this case.

Assume next that there exists $\epsilon_{0} > 0$ such that $\Hd (Y \cap W) \leq u - \epsilon_{0}$ for all $(d - 1)$-planes $W \subset \R^{d}$. Then, pick $d - 2 < t < \Hd X$ and $u - \epsilon_{0} < s < \Hd Y$ satisfying
\begin{equation}\label{form81} s > (d - 1) - \tfrac{1}{d - 1} - \eta \quad \text{and} \quad \Hd Y < s + \eta. \end{equation}
Let $\mu,\nu \in \mathcal{P}(\R^{d})$ with $I_{t}(\mu) < \infty$, $I_{s}(\nu) < \infty$ and $\spt(\mu) \subset X$ and $\spt(\nu) \subset Y$. Note that $\nu(W) = 0$ for all $(d - 1)$-planes $W \subset \R^{d}$, because otherwise
\[
\Hd (Y \cap W) \geq \Hd(\spt(\nu) \cap W) \geq s > u - \epsilon_{0},
\]
contrary to our assumption. Also, $\Hd \spt(\nu) \leq \Hd Y < s + \eta$.

These observations nearly place us in a position where we can apply the case \eqref{form79} of the first part of the proof. We are only missing the information that $\mu(W) = 0$ for all $(d - 1)$-planes $W \subset \R^{d}$. So, let us now treat the special case where $\mu(W) > 0$ for some $(d - 1)$-plane $W \subset \R^{d}$.

Since $\nu(W) = 0$, there exists a compact set $K \subset Y \, \setminus \, W$ such that $\nu(K) > 0$. Now, since $W$ is a $(d - 1)$-plane, the family of radial projections $\{\pi_{x}\}_{x \in W}$ is "morally the same" as the family of orthogonal projections $\{P_{V}\}_{V \in \mathcal{G}(d,d - 1)}$. More precisely, there exist (dimension preserving) projective maps $F \colon \R^{d} \to \R^{d}$ and $x \mapsto V(x)$ from $W$ to $\mathcal{G}(d,d - 1)$ such that
\begin{equation}\label{form82} \Hd \pi_{x}(K) = \Hd P_{V(x)}(F(K)), \qquad x \in W. \end{equation}
We outsource the justification to Remark \ref{rem1}. By the Kaufman-Mattila exceptional set estimate for orthogonal projections \cite[Theorem 5.10]{MR3617376}, combined with \eqref{form82}, we have
\begin{displaymath} \Hd \{x \in W : \Hd \pi_{x}(K) < \sigma\} = \Hd \{V \in \mathcal{G}(d,d - 1) : \Hd P_{V}(F(K)) < \sigma\} \leq \sigma \end{displaymath}
for all $0 \leq \sigma < \Hd K$. Since $\Hd (X \cap W) \geq \Hd (\spt(\mu) \cap W) \geq s$, we therefore get
\begin{displaymath} \sup_{x \in X} \Hd \pi_{x}(Y) \geq \sup_{x \in X \cap W} \Hd \pi_{x}(K) \geq \Hd K \geq s. \end{displaymath}
This proves our claim by letting $s \nearrow \Hd Y$.

We have now reduced the proof to the case where $\mu(W) = 0 = \nu(W)$ for all $(d - 1)$-planes $W \subset \R^{d}$. Since also $I_{t}(\mu) < \infty$ and $I_{s}(\nu) < \infty$, where $s$ satisfies \eqref{form81}, we may apply the "thin tubes" statement (case \eqref{form79}) at the beginning of the proof. The conclusion is that $(\mu,\nu)$ has $\sigma$-thin tubes for all $0 \leq \sigma < \min\{s,t\}$. Therefore $\sup_{x \in X} \Hd \pi_{x}(Y) \geq \min\{s,t\}$, and the proof is completed by letting $s \nearrow \Hd Y$ and $t \nearrow \Hd X$.
\end{proof}

\begin{remark}\label{rem1} Let $W := \R^{d - 1} \times \{0\} \subset \R^{d}$. The purpose of this remark is to justify the (well-known) fact that the family of radial projections $\{\pi_{x}\}_{x \in W}$ is "morally the same" as the family of orthogonal projections $\{P_{V}\}_{V \in \mathcal{G}(d,d - 1)}$, and in particular \eqref{form82}. To see this, let $F \colon \R^{d} \, \setminus \, W \to \R^{d}$ be the projective transformation
\begin{displaymath} F(\bar{x},x_{d}) := \frac{(\bar{x},1)}{x_{d}}, \qquad \bar{x} \in \R^{d - 1}, \, x_{d} \neq 0. \end{displaymath}
For $w \in \R^{d - 1}$ and $e \in S^{d - 1} \, \setminus \, W$, let $\ell_{w}(e) := (w,0) + \spa(e)$. The family
\begin{displaymath} \mathcal{L}(w) := \{\ell_{w}(e) : e \in S^{d - 1} \, \setminus \, W\}, \qquad w \in \R^{d - 1}, \end{displaymath}
then contains all the lines passing through $(w,0) \in W$ which are not contained in $W$. These lines are the fibres of the radial projections $\pi_{(w,0)}$, restricted to $\R^{d} \, \setminus \, W$. Now, a straightforward calculation shows that
\begin{displaymath} F(\ell_{w}(e)) = L_{e}(w), \qquad w \in \R^{d - 1}, \, e \in S^{d - 1} \, \setminus \, W, \end{displaymath}
where $L_{e}(w) = \spa(w,1) + (\bar{e}/e^{d},0)$. Therefore, $F$ transforms the lines in $\mathcal{L}(w)$ passing through $(w,0) \in W$ into lines parallel to the vector $(w,1)$. Since $(\bar{e}/e^{d},0)$ can take any value in $W$, in fact $\{F(\ell) : \ell \in \mathcal{L}(w)\}$ consists of all lines parallel to $(w,1)$. These lines are the fibres of the orthogonal projection $\pi_{V(w)}$ to $V(w) := (w,1)^{\perp} \in \mathcal{G}(d,d - 1)$. It follows from these observations (with a little more effort) that
\begin{displaymath} \Hd \pi_{(w,0)}(K) = \Hd \pi_{V(w)}(F(K)), \qquad w \in \R^{d - 1}, \, K \subset \R^{d} \, \setminus \, W, \end{displaymath}
as we claimed in \eqref{form82}. \end{remark}

Finally, we prove the case $k = d -1$ of Theorem \ref{gd1}.

\begin{proof}[Proof of Theorem \ref{gd1} for $k = d - 1$] We first claim the following. Let $d - 1 < t \leq d$ and $d - 2 < s \leq d$. Let $\mu,\nu$ be finite Borel measures on $\R^{d}$ with disjoint supports, $I_{s}(\mu) < \infty$, and $I_{t}(\nu) < \infty$. Then $(\mu,\nu)$ has $\sigma$-thin tubes for
\begin{displaymath} 0 \leq \sigma < \min\{s + t - (d - 1),d - 1\}. \end{displaymath}
Since $s,t > d - 2$, the same argument as in the previous proof (applying the Marstrand-Mattila projection and slicing theorems) yields that
\[
I_{s - (d - 2)}(\nu_{V,x})<\infty, \quad I_{t - (d - 2)}(\mu_{V,x})<\infty \quad\text{for } (\mu\times\gamma_{d,2})\text{-almost all }(x,V).
\]
We then apply the planar case, Theorem \ref{mainThinTubes} to the measures $\mu_{V,x}$ and $\nu_{V,x}$ (the hypotheses are evidently valid). The conclusion is that $(\mu_{V,x},\nu_{V,x})$ has $\sigma$-thin tubes for all
\begin{displaymath} 0 \leq \sigma < \min\{[s - (d - 2)] + [t - (d - 2)] - 1,1\}. \end{displaymath}
Now, observe that
\begin{displaymath} [s - (d - 2)] + [t - (d - 2)] - 1 = (s + t - (d - 1)) - (d - 2). \end{displaymath}
Consequently, it follows from Theorem \ref{thm:thin-tubes} that $(\mu,\nu)$ have $\sigma$-thin tubes for all $0 \leq \sigma \leq \min\{s + t - (d - 1),d - 1\}$, as we claimed.

The rest is standard, but let us spell out the details anyway. Assume, to reach a contradiction, that there exists a Borel set $Y \subset \R^{d}$ with $d - 1 < \Hd Y \leq d$ such that
\begin{displaymath} \Hd \{x \in \R^{d} : \Hd \pi_{x}(Y) < \sigma - \epsilon_{0}\} > \max\{(d - 1) + \sigma - \Hd Y,0\} \end{displaymath}
for some $\sigma \in [0,d - 1)$ and $\epsilon_{0} > 0$. We may infer from this counter assumption that $\sigma > \Hd Y - 1$ because $\{x \in \R^{d} : \Hd \pi_{x}(Y) < \Hd Y - 1\} = \emptyset$. Indeed, $Y\setminus\{x\}\subset \pi_{x}(Y)\times (0,\infty)$ when viewed in polar coordinates centred at $x$, and so $\Hd Y\le \Hd(\pi_{x}Y)+1$.   Therefore
\begin{equation}\label{form83} s := (d - 1) + \sigma - \Hd Y > d - 2. \end{equation}
Let $\mu \in \mathcal{P}(\R^{d})$ with $\spt(\mu) \subset \{x : \Hd \pi_{x}(Y) < \sigma - \epsilon_{0}\}$ with $I_{s}(\mu) < \infty$. Let
\begin{displaymath} \max\{d - 1,\Hd Y - \epsilon_{0}/2\} < t < \Hd Y, \end{displaymath}
and let $\nu \in \mathcal{P}(\R^{d})$ with $\spt(\nu) \subset Y$ with $I_{t}(\nu) < \infty$. As a small detail, we may arrange at this point that $\spt(\mu) \cap \spt(\nu) = \emptyset$ by restricting $\mu,\nu$ to disjoint balls with positive measure. By the claim established in the first part of the proof, it follows that $(\mu,\nu)$ has $\xi$-thin tubes for all $0 \leq \xi < \min\{s + t - (d - 1),d - 1\}$. We pick
\begin{displaymath} \xi > \min\{s + t - (d - 1) - \epsilon_{0}/2,d - 1 - \epsilon_{0}\} \stackrel{\eqref{form83}}{>} \sigma - \epsilon_{0}, \end{displaymath}
so that $(\mu,\nu)$ in particular has $(\sigma - \epsilon_{0})$-thin tubes. This implies that there exists a point $x \in \spt(\mu)$ with $\Hd \pi_{x}(Y) \geq \sigma - \epsilon_{0}$. This contradicts $\spt(\mu) \subset \{x : \Hd \pi_{x}(Y) < \sigma - \epsilon_{0}\}$, and completes the proof. \end{proof}

\def\cprime{$'$}


\end{document}